\documentclass[a4paper,12pt]{article}
\usepackage{latexsym}	
\usepackage{amsmath}
\usepackage{amsthm}	

\usepackage{graphicx}
\usepackage{txfonts}
\usepackage{bm}
\usepackage{color}
\usepackage{epic,eepic}

\newcommand{\finboxARX}{\finbox} 

\usepackage{geometry}
\numberwithin{equation}{section}

\newtheorem{theorem}{Theorem}[section]
\newtheorem{proposition}[theorem]{Proposition}
\newtheorem{corollary}[theorem]{Corollary}
\newtheorem{lemma}[theorem]{Lemma}
\newtheorem{remark}{Remark}[section]
\newtheorem{example}{Example}[section]
\newtheorem{claim}{Claim}[section]

\newcommand{\OMIT}[1]{{\bf [OMIT:} #1 \ {\bf --- end OMIT] }}  
   \renewcommand{\OMIT}[1]{}            

\newcommand{\RR}{{\mathbb{R}}}
\newcommand{\ZZ}{{\mathbb{Z}}}

\newcommand{\vecone}{{\bf 1}}
\newcommand{\veczero}{{\bf 0}}
\newcommand{\Lnat}{{L$^{\natural}$}}
\newcommand{\Mnat}{{M$^{\natural}$}}
\newcommand{\LLnat}{{L$^{\natural}_{2}$}}
\newcommand{\MMnat}{{M$^{\natural}_{2}$}}
\newcommand{\LL}{{L$_{2}$}}
\newcommand{\MM}{{M$_{2}$}}

\newcommand{\finbox}{\hspace*{\fill}$\rule{0.2cm}{0.2cm}$}
\newcommand{\todaye}{\the\year/\the\month/\the\day}

\begin{document}

\title{
Decomposition of an Integrally Convex Set 
\\ 
into a Minkowski Sum of 
\\
Bounded and Conic Integrally Convex Sets
}

\author{
Kazuo Murota%
\thanks{
The Institute of Statistical Mathematics,
Tokyo 190-8562, Japan; 
and
Faculty of Economics and Business Administration,
Tokyo Metropolitan University, 
Tokyo 192-0397, Japan,
murota@tmu.ac.jp}
\ and 
Akihisa Tamura%
\thanks{Department of Mathematics, Keio University, 
Yokohama 223-8522, Japan,
aki-tamura@math.keio.ac.jp}
}

\date{June 2023 / October 2023}

\maketitle

\begin{abstract}
Every polyhedron can be decomposed into a Minkowski sum (or vector sum) of 
a bounded polyhedron and a polyhedral cone.
This paper establishes similar statements
for some classes of discrete sets in discrete convex analysis,
such as integrally convex sets,
\Lnat-convex sets, and \Mnat-convex sets.
\end{abstract}

{\bf Keywords}:
Discrete convex analysis,  
Integrally convex set,
\Lnat-convex set, 
\Mnat-convex set,
Minkowski sum,
Characteristic cone.





\newpage



\section{Introduction}
\label{SCintro}

As is well known, every polyhedron can be decomposed into a Minkowski sum (or vector sum) of 
a bounded polyhedron and a polyhedral cone.
The objective of this paper is to establish 
similar decomposition theorems
with additional features of integrality and discrete convexity
using concepts from discrete convex analysis
\cite{Fuj05book,Mdca98,Mdcasiam,Mbonn09,Mdcaeco16}.
Emphasis is laid on integrally convex sets.
This notion in discrete convex analysis
is equivalent, via convex hull, 
to that of box-integer polyhedra in the theory of polyhedra \cite{Sch86,Sch03}
(see Proposition~\ref{PRintpolyIC} for the precise statement).

Integral convexity is a fundamental concept 
introduced by Favati--Tardella \cite{FT90} for functions
on the integer lattice $\ZZ\sp{n}$,
and integrally convex sets are defined in 
\cite[Section 3.4]{Mdcasiam} as the set version of  
integral convexity; see Section~\ref{SCprelimIC} for the precise definition.
Integral convexity encompasses almost all kinds of discrete convexity
proposed so far, 
such as \Lnat-convexity, \Mnat-convexity, \MMnat-convexity, and multimodularity
\cite{Mdcasiam}.
A discrete fixed point theorem was formulated by 
Iimura--Murota--Tamura \cite{IMT05} 
in terms of integrally convex sets 
(see also \cite[Section 11.9]{Mbonn09},
\cite[Section 13.1]{Mdcaeco16} for expositions).
Mathematical properties of integrally convex sets and functions
have been clarified in recent studies
(Moriguchi--Murota \cite{MM19projcnvl},
Moriguchi--Murota--Tamura--Tardella \cite{MMTT19proxIC},
Murota--Tamura \cite{MT20subgrIC,MT22ICfenc}).
The reader is referred to
Murota--Tamura \cite{MT23ICsurv}
for a recent comprehensive survey on integral convexity.

For any sets $S_{1}$, $S_{2} \subseteq \RR\sp{n}$,
we denote their {\em Minkowski sum} (or {\em vector sum})
by $S_{1}+S_{2}$, that is,
\begin{equation*}  
S_{1}+S_{2} = \{ x + y \mid x \in S_{1}, \  y \in S_{2} \} .
\end{equation*}
Let $P \subseteq \RR\sp{n}$ be a polyhedron.
A fundamental fact in the theory of polyhedra says
that it can be represented as $P= Q + C$ 
with a bounded polyhedron $Q$ and a polyhedral cone $C$
(see Section~\ref{SCprelimPolyh} for details).
In this decomposition, the cone $C$ is uniquely determined from $P$,
coinciding with the characteristic (or recession) cone of $P$,
whereas there is some degree of freedom in the choice of $Q$.
We are interested in integrality and discrete convexity
in this decomposition, and our contribution consists of two phases.

In the first phase we consider a box-integer polyhedron $P$ and 
impose an additional condition that $Q$ and $C$ be box-integer polyhedra.
Our first main result, Theorem~\ref{THboxintDec}, states that
this is indeed possible.
Furthermore, it is shown in Theorem~\ref{THlnatmnatDecR} that
if $P$ is an \Lnat-convex (resp., \Mnat-convex) polyhedron,
then we can impose that 
$Q$ and $C$ be \Lnat-convex  (resp., \Mnat-convex).
A technical challenge in establishing 
Theorem~\ref{THboxintDec} for box-integer polyhedra
stems from the lack of `outer description' 
of box-integer polyhedra in terms of inequality systems.
In contrast, inequality systems are available for
\Lnat-convex and \Mnat-convex polyhedra,
which makes the proof of Theorem~\ref{THlnatmnatDecR}
shorter and more transparent.

In the second phase we are concerned with discrete sets $S \subseteq \ZZ\sp{n}$.
Our second main result, 
Theorem~\ref{THicsetDecZ},
states that an integrally convex set 
$S$ can be represented as $S= T + G$ 
with a bounded integrally convex set $T$ 
and a `conic' integrally convex set $G$.
Furthermore, it is shown in Theorem~\ref{THlnatmnatDecZ} that
if $S$ is an \Lnat-convex (resp., \Mnat-convex) set,
then we can impose that 
$T$ and $G$ be \Lnat-convex  (resp., \Mnat-convex).
A technical challenge in the second phase 
is to overcome the well-known difficulty of discreteness 
in the Minkowski summation.
Namely, for discrete sets 
$S_{1}, S_{2} \subseteq \ZZ^{n}$,
the Minkowski sum
$S_{1} + S_{2}$ 
may possibly be different from 
$(\overline{S_{1}} + \overline{S_{2}}) \cap \ZZ\sp{n}$
(see Figure \ref{FGminkowhole} of Example \ref{EXicdim2sumhole}
 for a concrete example).
The possibility of
$(\overline{S_{1}} + \overline{S_{2}}) \cap \ZZ\sp{n} \ne S_{1} + S_{2}$
prevents us to derive the decomposition theorem 
for integrally convex sets 
as a corollary of Theorem~\ref{THboxintDec} for box-integer polyhedra.

\medskip

This paper is organized as follows. 
Section~\ref{SCprelim} is devoted to preliminaries 
on polyhedra and integrally convex sets.
The main results are described in Section~\ref{SCresult}.
Section~\ref{SCresultBI} deals with
subsets of $\RR\sp{n}$ such as
box-integer polyhedra, 
\Lnat-convex polyhedra, and \Mnat-convex polyhedra,
while
Section~\ref{SCresultIC} treats 
subsets of $\ZZ\sp{n}$ such as
integrally convex sets,
\Lnat-convex sets, and \Mnat-convex sets.
The proofs are given in Section~\ref{SCproof},
and Section~\ref{SCconcl} concludes the paper.

\section{Preliminaries}
\label{SCprelim}

\subsection{Polyhedra}
\label{SCprelimPolyh}

A subset $P$ of $\RR\sp{n}$ is called a 
{\em polyhedron}
if it is described by a finite number of linear inequalities,
that is,
$P = \{ x \mid Ax \leq b \}$ for some matrix $A$ and a vector $b$.
In this paper we always assume that a polyhedron is nonempty.
A subset $Q$ of $\RR\sp{n}$ is called a {\em polytope}
if it is the convex hull of a finite number of points,
that is, $Q = \overline{S}$ for a finite subset $S$ of $\RR\sp{n}$,
where $\overline{S}$ denotes the convex hull of $S$.
It is known that a polytope is nothing but a bounded polyhedron. 
A subset $C$ of $\RR\sp{n}$ is called a 
{\em cone}
if $d \in C$ implies $\lambda d \in C$ for all $\lambda \geq 0$.
We follow \cite{Sch86,Sch03} for terminology about polyhedra.

Let $P$ be a polyhedron.
The {\em characteristic cone} of $P$, 
denoted by ${\rm char.cone}\, P$,
is the polyhedral cone given by
\begin{equation} \label{charconedef}
{\rm char.cone}\, P = \{ d \mid x+d \in P \mbox{ for all $x$ in $P$}  \}.
\end{equation}
The characteristic cone is also called the {\em recession cone}.
The following are basic facts about the characteristic cone:
\begin{align}
&
d \in {\rm char.cone}\, P 
\nonumber
\\ & \iff
\mbox{there is an $x$ in $P$ such that
  $x + \lambda d \in P$ for all $\lambda \geq 0$},
\label{charconeLambda}
\\ &
d \in {\rm char.cone}\, P
\nonumber
\\ & \iff
\mbox{for all $x$ in $P$, it holds that $x + \lambda d \in P$ for all $\lambda \geq 0$},
\label{repLMtam1=1005}
\\ &
P + {\rm char.cone}\, P = P,
\label{charconePCP}
\\ &
\mbox{If $P = \{ x \mid Ax \leq b \}$, 
then \ ${\rm char.cone}\, P = \{ d \mid A d \leq 0 \}$}.
\label{charconeAy0}
\end{align}

The following is a fundamental theorem, stating that 
a polyhedron can be decomposed 
into a Minkowski sum of a polytope and a cone.

\begin{proposition}[Decomposition theorem for polyhedra] \label{PRdecthmP} 
\quad 

\noindent
{\rm (1)}
 Every polyhedron $P$ can be represented as 
$P= Q + C$ 
with some polytope $Q$ and polyhedral cone $C$.

\noindent
{\rm (2)}
If $P =  Q + C$, with $Q$ a polytope and $C$ a polyhedral cone, 
then  $P$ is a polyhedron and $C= {\rm char.cone}\, P$.
\finboxARX
\end{proposition}

It is emphasized that the choice of the polytope $Q$ in $P= Q + C$  is not unique,
while $C$ is uniquely determined by $P$ as stated in (2).

A polyhedron is said to be
{\em rational}
if it is described by a finite number of linear inequalities with
rational coefficients.
A polyhedron $P$ is an {\em integer polyhedron}
if $P=\overline{P \cap \ZZ\sp{n}}$, i.e., if
it coincides with the convex hull
of the integer points contained in it,
or equivalently, if
$P$ is rational and each face of $P$ contains an integer vector. 
A polyhedron $P$
is called {\em box-integer} if 
$P \cap \{ x \in \RR\sp{n} \mid l   \leq x \leq  u \}$
is an integer polyhedron for each choice of 
integer vectors $l , u \in \ZZ\sp{n}$ with $l  \leq u$
(\cite[Section~5.15]{Sch03}).
We call a subset $B$ of $\RR\sp{n}$
an {\em integral box}
if 
$B = \{ x \in \RR\sp{n} \mid l   \leq x \leq  u \}$
for some integer vectors $l , u \in \ZZ\sp{n}$ with $l  \leq u$.

\subsection{Integrally convex sets}
\label{SCprelimIC}

In this section we introduce the concept of
integrally convex sets, as defined in 
\cite[Section 3.4]{Mdcasiam},
and discuss subtleties related to the Minkowski sum of
integrally convex sets.
The reader is referred to 
Murota--Tamura \cite{MT23ICsurv}
for technical details of integral convexity 
including the most recent results.

For $x \in \RR^{n}$
the {\em integral neighborhood} of $x$
is defined by
\begin{equation} \label{Nxdef}
N(x) = \{ z \in \mathbb{Z}^{n} \mid | x_{i} - z_{i} | < 1 \ (i=1,2,\ldots,n)  \} .
\end{equation}
It is noted that 
strict inequality ``\,$<$\,'' is used in this definition
and 
$N(x)$ admits an alternative expression
\begin{equation}  \label{intneighbordeffloorceil}
N(x) = \{ z \in \ZZ\sp{n} \mid
\lfloor x_{i} \rfloor \leq  z_{i} \leq \lceil x_{i} \rceil  \ \ (i=1,2,\ldots, n) \} ,
\end{equation}
where, for $t \in \RR$ in general, 
$\left\lfloor  t  \right\rfloor$
denotes the largest integer not larger than $t$
(rounding-down to the nearest integer)
and 
$\left\lceil  t   \right\rceil$ 
is the smallest integer not smaller than $t$
(rounding-up to the nearest integer).
That is,
$N(x)$ consists of all integer vectors $z$ 
between
$\left\lfloor x \right\rfloor 
=( \left\lfloor x_{1} \right\rfloor ,\left\lfloor x_{2} \right\rfloor , 
  \ldots, \left\lfloor x_{n} \right\rfloor)$ 
and 
$\left\lceil x \right\rceil 
= ( \left\lceil x_{1} \right\rceil, \left\lceil x_{2} \right\rceil,
   \ldots, \left\lceil x_{n} \right\rceil)$.

Let $S$ be a subset of $\ZZ\sp{n}$ and recall that $\overline{S}$
denotes the convex hull of $S$.
As is well known, $\overline{S}$ coincides with the set of 
all convex combinations of (finitely many) elements of $S$.
For any real vector $x \in \RR^{n}$,
we call the convex hull of $S \cap N(x)$ 
the {\em local convex hull} of $S$ around $x$.
A nonempty set
$S \subseteq \ZZ^{n}$ is said to be 
{\em integrally convex} if
the union of the local convex hulls $\overline{S \cap N(x)}$ over $x \in \RR^{n}$ 
is convex.
In other words, a set $S \subseteq \ZZ^{n}$ is called 
integrally convex if
\begin{equation}  \label{icsetdef0}
 \overline{S} = \bigcup_{x \in \RR\sp{n}} \overline{S \cap N(x)}.
\end{equation}
This condition is equivalent to saying that
every point $x$
in the convex hull of $S$ is contained in the convex hull of 
$S \cap N(x)$, i.e.,
\begin{equation}  \label{icsetdef1}
 x \in \overline{S} \ \Longrightarrow  x \in  \overline{S \cap N(x)} .
\end{equation}
Obviously, every subset of $\{ 0, 1\}\sp{n}$ is integrally convex.

We say that a set
$S \subseteq \ZZ\sp{n}$ is
{\em hole-free} if
\begin{equation}  \label{icsetholefree}
  S = \overline{S} \cap \ZZ\sp{n} .
\end{equation}
It is known that an integrally convex set is hole-free;
see \cite[Proposition~2.2]{MT23ICsurv} for a formal proof.
It is also known 
that the convex hull of an integrally convex set is 
a polyhedron (Murota--Tamura \cite[Section 4.1]{MT20subgrIC}).
However, no characterization is known 
about the inequality systems to describe integrally convex sets.

The concept of integrally convex sets is
closely related (or essentially equivalent) to 
that of box-integer polyhedra as follows.

\begin{proposition}[{\cite[Section 2.2]{Mopernet21}}] \label{PRintpolyIC} 
If a set $S \subseteq \ZZ^{n}$ is integrally convex, then its convex hull
$\overline{S}$ is a box-integer polyhedron
and $S = \overline{S} \cap \ZZ\sp{n}$.
Conversely, if $P$ is a box-integer polyhedron, then 
$P \cap \ZZ\sp{n}$ is an integrally convex set and 
$P = \overline{P \cap \ZZ\sp{n}}$.
\finboxARX
\end{proposition}

Minkowski summation is an intriguing operation in discrete setting.
For two (discrete) sets
$S_{1}, S_{2} \subseteq \ZZ\sp{n}$,
in general, we have
\begin{equation} \label{minkowZminkowR1}
\overline{ S_{1} + S_{2}  }
= 
\overline{ S_{1} } + \overline{ S_{2} } 
\end{equation}
(see, e.g., \cite[Proposition 3.17(4)]{Mdcasiam}).
In contrast, the naive looking relation
\begin{equation} \label{convminkowsumG}
   S_{1}+S_{2}  = ( \overline{S_{1}+ S_{2}}) \cap \ZZ\sp{n} 
\end{equation}
is not always true, as Example~\ref{EXicdim2sumhole} below shows.

\begin{example}[{\cite[Example 3.15]{Mdcasiam}}] \rm \label{EXicdim2sumhole}
The Minkowski sum of
$S_{1} = \{ (0,0), (1,1) \}$
and
$S_{2} = \{ (1,0), (0,1) \}$
is equal to 
$S_{1}+S_{2} = \{ (1,0), (0,1), (2,1), (1,2) \}$,
for which
$(1,1) \in (\overline{S_{1}+S_{2}}) \setminus (S_{1}+S_{2})$.
That is, the Minkowski sum $S_{1}+S_{2}$ has a `hole' at $(1,1)$.
See Figure~\ref{FGminkowhole}.
\finbox
\end{example}

\begin{figure}
\centering
\includegraphics[height=30mm]{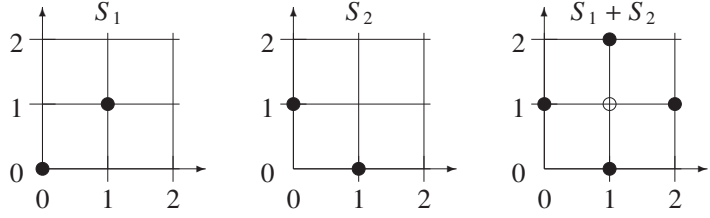}
\caption{Minkowski sum of discrete sets}
\label{FGminkowhole}
\end{figure}

It may be said that 
if \eqref{convminkowsumG} is true
for some class of discrete convex sets, 
this equality captures a certain essence of the discrete convexity in question.
For example,
\eqref{convminkowsumG} is true for two \Mnat-convex sets,
since the Minkowski sum of two \Mnat-convex sets remains to be \Mnat-convex
(\cite[Section 4.6]{Mdcasiam}, \cite[Section 3.5]{Msurvop21}).
The identity \eqref{convminkowsumG} also holds for two \Lnat-convex sets,
since the Minkowski sum of two \Lnat-convex sets is integrally convex
\cite[Theorem 8.42]{Mdcasiam},
although it is not necessarily \Lnat-convex.

For the Minkowski sum of integrally convex sets $S_{1}$ and $S_{2}$,
we observe the following.
\begin{itemize}
\item
$S_{1} + S_{2}$ may have a `hole', that is,
\eqref{convminkowsumG} may fail
(see Example~\ref{EXicdim2sumhole}).

\item
$S_{1} + S_{2}$ may not be integrally convex
(see Example~\ref{EXicdim2sumhole}).

\item
$\overline{(S_{1} + S_{2})} \cap \ZZ\sp{n}$ 
may not be integrally convex
(see Example~\ref{EXicdim3sum} below).
\end{itemize}

\begin{example} \rm \label{EXicdim3sum}
Consider
$S_{1} = \{ (1,0,0), (0,1,0), (0,0,1) \}$
and
$S_{2} = \{ (0,0,0),  \allowbreak (1,1,1) \}$.
Their Minkowski sum is given by 
$S_{1}+S_{2} = \{ (1,0,0), (0,1,0), 
      \allowbreak  (0,0,1), \allowbreak  (2,1,1), 
      \allowbreak  (1,2,1), (1,1,2) \}$.
Let 
$S = \overline{(S_{1} + S_{2})} \cap \ZZ\sp{3}$
and consider
$x = [ (1,0,0) + (1,1,2) ] /2  = (1, 1/2, 1)$
belonging to $\overline{S}$.
We have
$N(x) = \{ (1,0,1),  \allowbreak  (1,1,1)  \}$,
$N(x) \cap S = \{ (1,1,1)  \}$,
and $x \notin \overline{N(x) \cap S}$.
Thus the condition \eqref{icsetdef1} for integral convexity of $S$ is violated.
This example also shows that the Minkowski sum 
of box-integer polyhedra is not necessarily box-integer;
see also Remark~\ref{RMboxintMinkow}. 
\finbox
\end{example}

Discrepancy between
$S_{1}+S_{2}$ and $\overline{S_{1}+ S_{2}}$
has attracted considerable attention 
in (ordinary) convex analysis,
leading to the Shapley--Folkman theorem,
which has  applications in economics, optimization, etc.
A recent paper \cite{MT23shafo} of the present authors
shows a Shapley--Folkman-type theorem for integrally convex sets.

\section{Results}
\label{SCresult}

\subsection{Decomposition of box-integer polyhedra}
\label{SCresultBI}

In this section we describe our first main result
(Theorem~\ref{THboxintDec}),
a decomposition theorem for box-integer polyhedra.
The proof of this theorem relies
on the following technical results
(in their equivalent reformulations in 
Propositions \ref{PRreconeICgenerator} and \ref{PRreconeIC}; 
see Figure~\ref{FGprfstruct} in Section \ref{SCproof}).

\begin{proposition} \label{PRreconeBIgenerator} 
The characteristic cone of a box-integer polyhedron 
is generated by $\{ -1,0,  \allowbreak   +1 \}$-vectors.
\finboxARX
\end{proposition}

\begin{proposition} \label{PRreconeBI}
The characteristic cone of a box-integer polyhedron is box-integer.
\finboxARX
\end{proposition}

The proofs of these propositions
are quite long and involved,
probably because no characterization is known 
about inequality systems to describe box-integer polyhedra.
The proofs of Propositions \ref{PRreconeBIgenerator} and \ref{PRreconeBI} 
are given in Sections \ref{SCproofrecone} and \ref{SCproofPRreconeBI}, respectively.
Our decomposition theorem for box-integer polyhedra is as follows.

\begin{theorem} \label{THboxintDec}
 Every box-integer polyhedron $P$ can be represented as 
\begin{equation} \label{SQCboxintR}
 P = Q +  C 
\end{equation}
with a bounded box-integer polyhedron $Q$ and a box-integer polyhedral cone $C$.
\end{theorem}
\begin{proof}
By Proposition~\ref{PRdecthmP},
we can decompose $P$ as
$P = \hat{Q} +  C$,
where $\hat{Q}$ is a polytope and 
$C$ is the characteristic cone of $P$.
The cone $C$ is box-integer by Proposition~\ref{PRreconeBI}.
Take a bounded integral box $B$ containing $\hat{Q}$
and define
$Q := P \cap B$,
which is a bounded box-integer polyhedron.
Since
\[
Q = P \cap B = 
(\hat{Q} +  C) \cap B \supseteq \hat{Q} \cap B = \hat{Q},
\]
we obtain
\[
Q + C \supseteq  \hat{Q} + C = P.
\]
The reverse inclusion $Q + C \subseteq P$
follows from $Q \subseteq P$
and $P + C = P$ in \eqref{charconePCP}
as $Q + C \subseteq P + C = P$.
\end{proof}

\begin{remark} \rm \label{RMboxintMinkow}  
In view of Proposition~\ref{PRdecthmP}(2)
we may be tempted to imagine that 
if $Q$ is a bounded box-integer polyhedron
and $C$ is a box-integer polyhedral cone,
then $Q + C$ is a box-integer polyhedron.
But this is not the case.
A counterexample can be constructed from Example \ref{EXicdim3sum}.
Let $Q$ be the convex hull of  
$S_{1} = \{ (1,0,0), (0,1,0), (0,0,1) \}$
and
$C$ be the polyhedral cone generated by
$S'_{2} = \{ (1,1,1) \}$, 
that is, $C = \{ \lambda (1,1,1) \mid \lambda \geq 0 \}$. 
Both $Q$ and $C$ are box-integer, but $Q + C$ is not.
Indeed,
$T = (Q + C) \cap \ZZ\sp{3}$ is not integrally convex,
because
$x = [  (1,0,0) +  (0,0,1) ] /2 + (1,1,1)/2 
= (1, 1/2, 1) \in \overline{T}$,
$N(x) = \{ (1,0,1), (1,1,1)  \}$,
$N(x) \cap T = \{ (1,1,1)  \}$,
and $x \notin \overline{N(x) \cap T}$.
\finbox
\end{remark}

\begin{remark} \rm \label{RM01genNonIC}  
By Proposition~\ref{PRreconeBIgenerator},
a box-integer cone is generated by $\{ -1,0,  \allowbreak   +1 \}$-vectors,
but the latter property does not characterize a box-integer cone.
Consider the cone $C$ generated by 
$(1,1,0,1)$, $(0,1,1,1)$, $(1,0,1,1)$, that is,
\begin{align*} 
 C  
&= \{ x \mid x = \alpha_1 (1,1,0,1) + \alpha_2 (0,1,1,1) + \alpha_3 (1,0,1,1),
\ \ \alpha_1 , \alpha_2 , \alpha_3 \geq 0 \}
\nonumber \\ 
&= \{ x \mid x =(\alpha_1 + \alpha_3, \alpha_1 + \alpha_2, 
    \alpha_2 + \alpha_3,   \alpha_1 + \alpha_2 + \alpha_3) :
\ \ \alpha_1 , \alpha_2 , \alpha_3 \geq 0 \}.
\end{align*} 
For  $\alpha_1 = \alpha_2 = \alpha_3 = 1/2$, we have
$x = (1,1,1,3/2)$
and 
$N(x) = \{ (1,1,1,2), \allowbreak  (1,1,1,1) \} $.
But
$ (1,1,1,2) \notin C$ and $(1,1,1,1) \notin C$,
and hence
$N(x) \cap C = \emptyset$.
This shows that
$C \cap \ZZ^4$ is not integrally convex, and hence
$C$ is not box-integer.
\finbox
\end{remark}

Theorem~\ref{THboxintDec} can be adapted to 
some classes of integer polyhedra
treated in discrete convex analysis,
such as \Lnat-convex and \Mnat-convex polyhedra.
An {\em \Lnat-convex polyhedron} 
is, by definition, an integer polyhedron obtained as the convex hull of
an \Lnat-convex set.
It is known that 
an \Lnat-convex polyhedron $P$ can be described as 
\begin{equation} \label{ineqLnat}
P = \{ x \in \RR\sp{n} \mid 
 l_{i} \leq  x_{i} \  (i \in I), 
 x_{j} \leq u_{j} \ (j \in J), 
 x_{j} - x_{i} \leq d_{ij} \  ((i,j) \in E)
 \}
\end{equation}
for some 
$I, J \subseteq \{ 1,2,\ldots, n\}$,
$E \subseteq \{ 1,2,\ldots, n\} \times \{ 1,2,\ldots, n\}$,
$l_{i} \in \ZZ$ $(i \in I)$,
$u_{j} \in \ZZ$ $(j \in J)$,
and 
$d_{ij} \in \ZZ$ $((i,j) \in E)$,
and the converse is also true.
An {\em \Lnat-convex cone} means an \Lnat-convex polyhedron that is a cone.
An {\em \Mnat-convex polyhedron} is a synonym of an integral generalized polymatroid,
and hence an \Mnat-convex polyhedron $P$ is described as
\begin{equation} \label{ineqMnat}
P = \{ x \in \RR\sp{n} \mid \mu(X) \leq x(X) \leq \rho(X) \ 
(\forall X \subseteq \{ 1,2,\ldots, n\}) \},
\end{equation}
where $x(X) = \sum_{i \in X} x_{i}$,
for a (strong or paramodular) pair of 
an integer-valued supermodular function $\mu$ 
and an integer-valued submodular function $\rho$
(cf., \cite[Section 3.5(a)]{Fuj05book}, \cite[Section 4.7]{Mdcasiam});
$\mu$ and $\rho$ are allowed to take $-\infty$ and $+\infty$, respectively.
An {\em \Mnat-convex cone} is defined in an obvious manner.
Other kinds of polyhedra
(such as 
\LLnat-convex polyhedron, \MMnat-convex polyhedron, and multimodular polyhedron)
are defined similarly from the corresponding notions for sets of integer vectors.
More precisely,  an \LLnat-convex set is defined as the Minkowski sum 
of two \Lnat-convex sets and an \LLnat-convex polyhedron is
the convex hull of an \LLnat-convex set,
implying that an \LLnat-convex polyhedron can also be defined as
the Minkowski sum of two \Lnat-convex polyhedra.
Similarly, an \MMnat-convex set is defined as the intersection of two \Mnat-convex sets 
and an \MMnat-convex polyhedron is
the convex hull of an \MMnat-convex set;
then it is known (cf., e.g., \cite[Theorem 4.22]{Mdcasiam}) 
that an \MMnat-convex polyhedron can also be defined as
the intersection of two \Mnat-convex polyhedra.

The adaptation of Theorem~\ref{THboxintDec}
to specific classes is given in Theorem~\ref{THlnatmnatDecR} below.
It should be clear that,
although \Lnat-convex polyhedra, etc., 
constitute subclasses of box-integer polyhedra,
Theorem~\ref{THboxintDec} does not imply 
the corresponding statements for these subclasses.
It is worth noting that the proofs 
for these special cases  
do not rely on Theorem~\ref{THboxintDec} and that they
are shorter and simpler
because of the inequality descriptions 
known for these special cases (see Murota \cite{Mdcasiam}, 
Moriguchi--Murota \cite[Table 1]{MM22L2poly},
Murota--Tamura \cite[Table 1]{MT23ICsurv}).

\begin{theorem} \label{THlnatmnatDecR}
\quad

\noindent
{\rm (1)} 
 Every \Lnat-convex polyhedron $P$ can be represented as 
$P= Q + C$ 
with a bounded \Lnat-convex polyhedron $Q$ and an \Lnat-convex cone $C$.

\noindent
{\rm (2)} 
 Every \LLnat-convex polyhedron $P$ can be represented as 
$P= Q + C$ 
with a bounded \LLnat-convex polyhedron $Q$ and an \LLnat-convex cone $C$.

\noindent
{\rm (3)} 
 Every \Mnat-convex polyhedron $P$ can be represented as 
$P= Q + C$ 
with a bounded \Mnat-convex polyhedron $Q$ and an \Mnat-convex cone $C$.
Similarly for an M-convex polyhedron $P$,
with $Q$ and $C$ being M-convex.

\noindent
{\rm (4)} 
 Every \MMnat-convex polyhedron $P$ can be represented as 
$P= Q + C$ 
with a bounded \MMnat-convex polyhedron $Q$ and an \MMnat-convex cone $C$.
Similarly for an \MM-convex polyhedron $P$, with $Q$ and $C$ being \MM-convex.

\noindent
{\rm (5)} 
 Every multimodular polyhedron $P$ can be represented as 
$P= Q + C$ 
with a bounded multimodular polyhedron $Q$ and a multimodular cone $C$.
\end{theorem}

\begin{proof}
(1) 
The proof of Theorem~\ref{THboxintDec} 
can be adapted to an \Lnat-convex polyhedron 
on the basis of the following properties of an \Lnat-convex polyhedron.
\begin{enumerate}
\item
The characteristic cone of an \Lnat-convex polyhedron is \Lnat-convex.

\item
The intersection of an \Lnat-convex polyhedron 
with an integral box is \Lnat-convex.
\end{enumerate}

We can prove the first statement by making use of the fact 
that an \Lnat-convex polyhedron $P \subseteq \RR\sp{n}$
is described as \eqref{ineqLnat}.
It follows from \eqref{ineqLnat} and  \eqref{charconeAy0} that
the characteristic cone of $P$ is given by
$C = \{ x \mid 
 0 \leq  x_{i} \  (i \in I), 
 x_{j} \leq 0 \ (j \in J), 
 x_{j} - x_{i} \leq 0 \  ((i,j) \in E)
 \}$,
which is also an \Lnat-convex polyhedron.
The second statement also follows from \eqref{ineqLnat}.  
We consider the decomposition
$P = \hat{Q} +  C$
in Proposition~\ref{PRdecthmP},
take a bounded integral box $B$ containing $\hat{Q}$,
and define $Q := P \cap B$,
for which we can show $P = Q + C$
as in the proof of Theorem~\ref{THboxintDec}. 

(2)--(5)
These cases are proved in Section~\ref{SCproofTHlnatmnatDecR}
by using a unified proof scheme consistent with the case of (1).
\end{proof}

\begin{remark} \rm \label{RMnoLpoly}
Theorem~\ref{THlnatmnatDecR}(1) gives a decomposition of an 
\Lnat-convex polyhedron.
However, we cannot obtain a similar statement for an L-convex polyhedron,
simply because there is no bounded L-convex polyhedron.
Note that an L-convex polyhedron $P$ 
has the invariance in the direction of $\vecone = (1,1,\ldots,1)$
in the sense that $x \in P$ implies $x + \lambda \vecone \in P$
for all $\lambda \in \RR$. 
Similarly, there is no bounded \LL-convex polyhedron.
\finbox
\end{remark}

\begin{remark} \rm \label{RMnonintLM}
In each case of Theorem~\ref{THlnatmnatDecR},
the polyhedron $P$ is necessarily an integer polyhedron.
Recall that
we have defined $P$ to be an \Lnat-convex polyhedron
if it is the convex hull of an \Lnat-convex set 
$S$ $(\subseteq \ZZ\sp{n})$.
In the literature of discrete convex analysis,
the notion of \Lnat-convexity is generalized to 
non-integer polyhedra
(Murota--Shioura \cite{MS00poly}).
An \Lnat-convex polyhedron (not necessarily integral)
is described by \eqref{ineqLnat}
with
$l_{i} \in \RR$ $(i \in I)$,
$u_{j} \in \RR$ $(j \in J)$,
and 
$d_{ij} \in \RR$ $((i,j) \in E)$.
For an \Lnat-convex polyhedron $P$ in this generalized sense,
we also obtain the decomposition $P= Q + C$. 
Similar generalizations are possible for 
\Mnat-convex polyhedra, etc., in 
(2)--(5) of Theorem~\ref{THlnatmnatDecR}.
 \finbox
\end{remark}

\subsection{Decomposition of integrally convex sets}
\label{SCresultIC}

Theorem~\ref{THboxintDec} for box-integer polyhedra
can be rephrased for integrally convex sets as follows.

\begin{corollary} \label{COicsetDecR}
The convex hull 
$\overline{S}$ of an integrally convex set
$S$ $(\subseteq \ZZ\sp{n})$
can be represented as 
\begin{equation} \label{SQCicsetR}
 \overline{S} = Q + C 
\end{equation}
with a polytope $Q$ and a polyhedral cone $C$
such that 
$Q \cap \ZZ\sp{n}$ and
$C \cap \ZZ\sp{n}$ are integrally convex.
\end{corollary}
\begin{proof}
Since $S$ is integrally convex,
$\overline{S}$ is a box-integer polyhedron
by Proposition~\ref{PRintpolyIC}.
By Theorem~\ref{THboxintDec} applied to $\overline{S}$
we obtain the decomposition \eqref{SQCicsetR},
where 
$Q$ is a bounded box-integer polyhedron
and 
$C$ is a box-integer cone.
Then $Q \cap \ZZ\sp{n}$ and $C \cap \ZZ\sp{n}$ are integrally convex
by Proposition~\ref{PRintpolyIC}.
\end{proof}

While the decomposition 
$\overline{S} = Q + C$
in \eqref{SQCicsetR}
is defined via embedding of $S$
into $\RR\sp{n}$,
our second main result
(Theorem~\ref{THicsetDecZ} below)
establishes a decomposition 
of an integrally convex set $S$
directly within $\ZZ\sp{n}$.
We emphasize the difference between
$\overline{S} = Q + C$ and
\begin{equation} \label{SQCicsetZ0}
S = (Q\cap \ZZ\sp{n}) + (C \cap \ZZ\sp{n}) .
\end{equation}
We can show 
``\eqref{SQCicsetZ0} $\Rightarrow$ \eqref{SQCicsetR}'' as
\[
\overline{S} 
= \overline{(Q\cap \ZZ\sp{n}) + (C \cap \ZZ\sp{n})}
= \overline{Q\cap \ZZ\sp{n}} +  \overline{C \cap \ZZ\sp{n}}
= Q + C ,
\]
where $\overline{ S_{1} + S_{2}  } = \overline{ S_{1} } + \overline{ S_{2} }$
in \eqref{minkowZminkowR1} is used.
However, the converse 
``\eqref{SQCicsetZ0} $\Leftarrow$ \eqref{SQCicsetR}''
is not always true (see Example \ref{EXdecRZcomp} below).
Thus, \eqref{SQCicsetZ0} is
(strictly) stronger than \eqref{SQCicsetR}.

To state the theorem we need to introduce a terminology.
We call a set  
$G$ $(\subseteq \ZZ\sp{n})$ 
a {\em conic set} 
if its convex hull $\overline{G}$ is a cone.
An integrally convex set $G$ is conic if and only if
$G = C \cap \ZZ\sp{n}$ for some box-integer cone $C$.

\begin{theorem} \label{THicsetDecZ}
Every integrally convex set $S$ 
$(\subseteq \ZZ\sp{n})$
can be represented as 
\begin{equation} \label{SQCicsetZ}
 S = T + G 
\end{equation}
with a bounded integrally convex set $T$ and a conic integrally convex set $G$.
\end{theorem}
\begin{proof}
The proof, to be given in Section~\ref{SCproofTHicsetDecZ},
is based on propositions equivalent to 
Propositions \ref{PRreconeBIgenerator} and \ref{PRreconeBI}.
\end{proof}

\begin{example} \rm \label{EXdecRZcomp}
We compare the decompositions 
in Corollary~\ref{COicsetDecR} and Theorem~\ref{THicsetDecZ}
for a simple two-dimensional example.
Let $S$ be an infinite subset of $\ZZ\sp{2}$
depicted at the top left of Figure~\ref{FGminkowIConZ},
which can be described, e.g., as
$S = \{ x \in \ZZ\sp{2} \mid 
  x_{1} + x_{2} \geq 1, \  | x_{1} - x_{2} | \leq 1   \}$.
This set $S$ is integrally convex,
and the convex hull $\overline{S}$ is a box-integer polyhedron 
described as
$\overline{S} = \{ x \in \RR\sp{2} \mid 
  x_{1} + x_{2} \geq 1, \  | x_{1} - x_{2} | \leq 1   \}$.
Let $Q$ be the line segment connecting $(1,0)$ and $(0,1)$
and $C$ be the semi-infinite line starting at $(0,0)$ 
and emanating in the direction of $(1,1)$.
Both $Q$ and $C$ are box-integer, and 
we obtain the decomposition $\overline{S} = Q + C$ 
in Corollary~\ref{COicsetDecR}.
The semi-infinite line $C$ is, in fact, the characteristic cone of $\overline{S}$.
Both $Q\cap \ZZ\sp{2}$ and $C \cap \ZZ\sp{2}$
are integrally convex, but the identity
$S = (Q\cap \ZZ\sp{2}) + (C \cap \ZZ\sp{2})$
in \eqref{SQCicsetZ0} fails, 
because of the `holes' in $(Q\cap \ZZ\sp{2}) + (C \cap \ZZ\sp{2})$
at $x = (t,t)$ for integers $t \geq 1$.
With the choice of
$T = (Q\cap \ZZ\sp{2}) \cup \{ (1,1) \} = \{ (1,0), (0,1), (1,1) \}$
and
$G = C \cap \ZZ\sp{2} = \{ (t,t) \mid t\geq 0, t \in \ZZ \}$,
we obtain the decomposition 
$S = T + G$ in Theorem~\ref{THicsetDecZ}.
Here both $T$ and $G$ are integrally convex.
\finbox
\end{example}

\begin{figure}
\centering
\includegraphics[width=0.75\textwidth,clip]{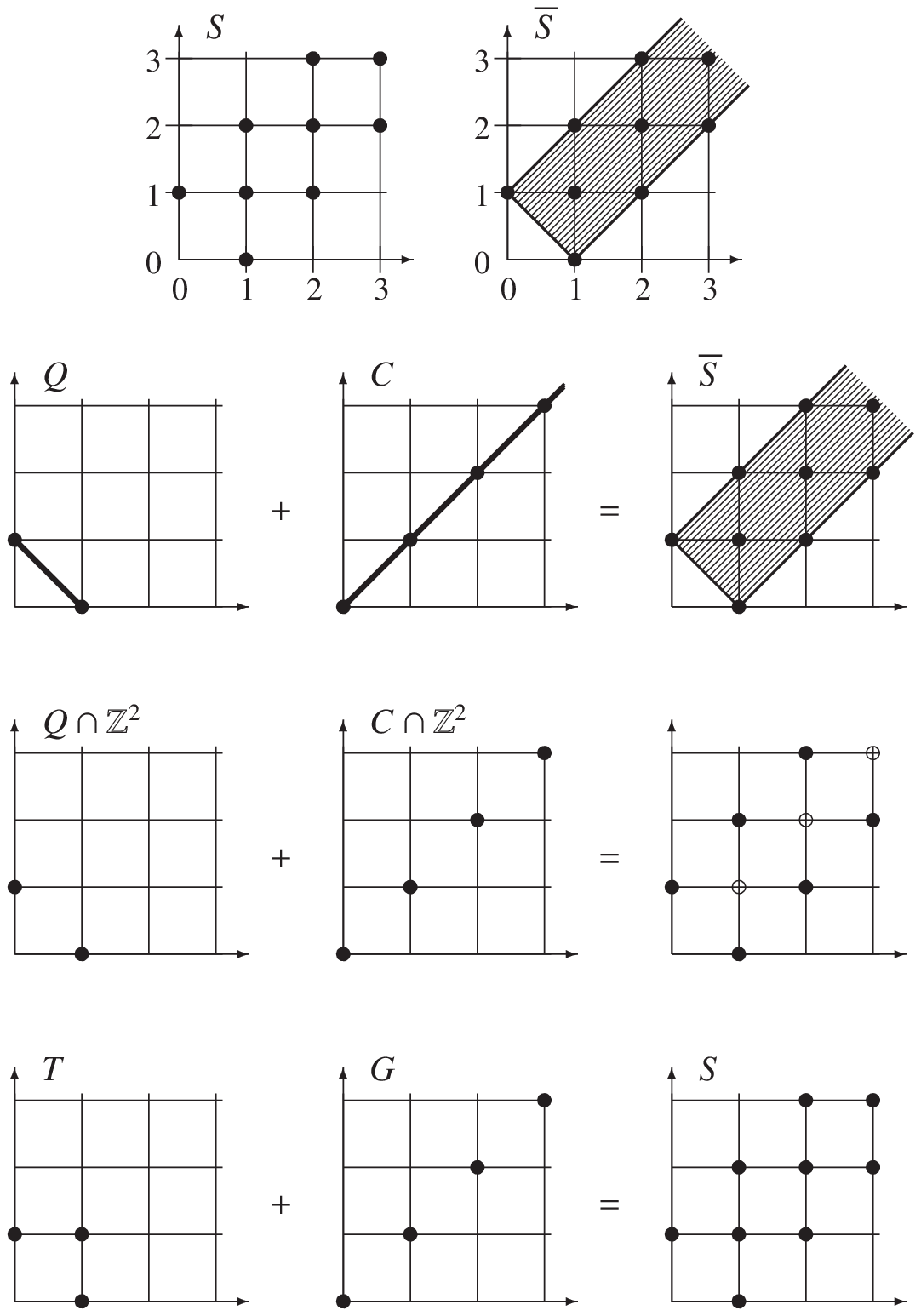}
\caption{$Q + C = \overline{S}$,
$(Q\cap \ZZ\sp{2}) + (C \cap \ZZ\sp{2}) \ne S$, and $T + G = S$}
\label{FGminkowIConZ}
\end{figure}

Theorem~\ref{THicsetDecZ} can be adapted to 
some classes of discrete convex sets in discrete convex analysis,
such as \Lnat-convex and \Mnat-convex sets
(see Murota \cite{Mdcasiam} for definitions of these concepts). 
The corresponding statements for these subclasses
are given in 
Theorem~\ref{THlnatmnatDecZ} below.
It is emphasized that Theorem~\ref{THlnatmnatDecZ}
does not follow from Theorem~\ref{THicsetDecZ}
(for general integrally convex sets)
nor from Theorem~\ref{THlnatmnatDecR} (for \Lnat-convex polyhedra, etc.).
Note that we have
$S, T, G \subseteq \ZZ\sp{n}$ in Theorem~\ref{THlnatmnatDecZ},
whereas 
$P, Q, C \subseteq \RR\sp{n}$
in Theorem~\ref{THlnatmnatDecR}.

\begin{theorem} \label{THlnatmnatDecZ}
\quad

\noindent
{\rm (1)} 
 Every \Lnat-convex set $S$ 
can be represented as 
$S= T + G$ 
with a bounded \Lnat-convex set $T$ and a conic \Lnat-convex set $G$.

\noindent
{\rm (2)} 
 Every \LLnat-convex set $S$ can be represented as 
$S= T + G$ 
with a bounded \LLnat-convex set $T$ and a conic \LLnat-convex set $G$.

\noindent
{\rm (3)} 
 Every \Mnat-convex set $S$ can be represented as 
$S= T + G$ 
with a bounded \Mnat-convex set $T$ and a conic \Mnat-convex set $G$.
Similarly for an M-convex set $S$,
with $T$ and $G$ being M-convex.

\noindent
{\rm (4)} 
 Every \MMnat-convex set $S$ can be represented as 
$S= T + G$ 
with a bounded \MMnat-convex set $T$ and a conic \MMnat-convex set $G$.
Similarly for an \MM-convex set $S$, with $T$ and $G$ being \MM-convex.

\noindent
{\rm (5)} 
 Every multimodular set $S$ can be represented as 
$S= T + G$ 
with a bounded multimodular set $T$ and a conic multimodular set $G$.
\end{theorem}

\begin{proof}
The proof is given in Section~\ref{SCproofTHlnatmnatDecZ}.
\end{proof}

\section{Proofs}
\label{SCproof}

The structure of the proofs
(dependence among propositions and theorems) is shown in the diagram 
in Figure~\ref{FGprfstruct}.

\begin{figure}
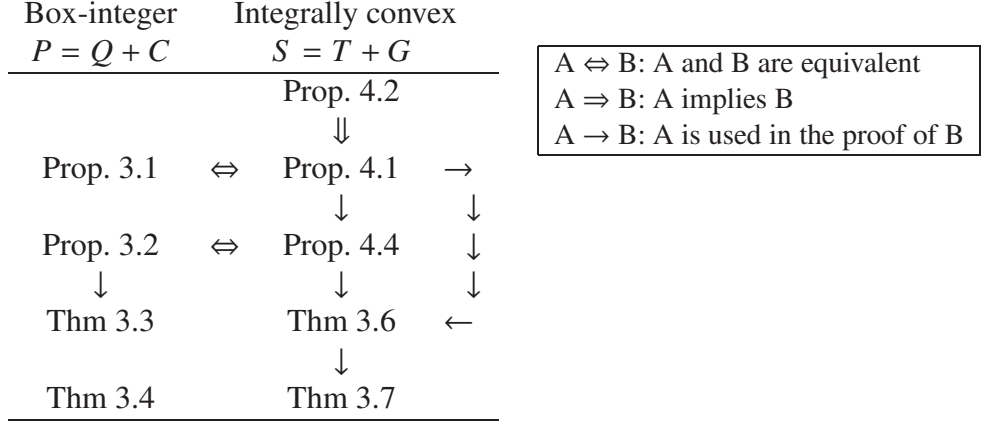

\centering
\begin{tabular}{cccr}
Box-integer &   \multicolumn{3}{c}{Integrally convex}  
\\
 $P = Q + C$  & &  $S = T + G$ 
\\ \hline
 & & Prop.~\ref{PRconeICgenerator} & 
\\
 &  & $\Downarrow$ & 
\\
Prop.~\ref{PRreconeBIgenerator} & $\Leftrightarrow$ & Prop.~\ref{PRreconeICgenerator} 
& $\rightarrow$ \
\\
 &  & $\downarrow$ & $\downarrow$ 
\\
 Prop.~\ref{PRreconeBI} & $\Leftrightarrow$ & Prop.~\ref{PRreconeIC} & $\downarrow$
\\
   $\downarrow$ &  &  $\downarrow$  & $\downarrow$ 
\\
Thm \ref{THboxintDec} &  &  Thm \ref{THicsetDecZ}  & $\leftarrow$ \  
\\
    &  &  $\downarrow$  & 
\\
 Thm \ref{THlnatmnatDecR}   &  & Thm \ref{THlnatmnatDecZ}  &  
\\ \hline
\end{tabular}
\quad
{\small 
\begin{tabular}{l}
\hline
\multicolumn{1}{|l|}{A $\Leftrightarrow$ B: A and B are equivalent} 
\\
\multicolumn{1}{|l|}{A $\Rightarrow$ B: A implies B} 
\\
\multicolumn{1}{|l|}{A $\rightarrow$ B: A is used in the proof of B} 
\\ \hline
\\ \multicolumn{1}{c}{}
\\ \multicolumn{1}{c}{}
\\ \multicolumn{1}{c}{}
\\ \multicolumn{1}{c}{}
\\ \multicolumn{1}{c}{}
\end{tabular}
}
\caption{Dependence among propositions and theorems}
\label{FGprfstruct}
\end{figure}

\subsection{Proof of Proposition~\ref{PRreconeBIgenerator}}
\label{SCproofrecone}

In this section we prove Proposition~\ref{PRreconeBIgenerator},
stating that the characteristic cone of a box-integer polyhedron 
is generated by $\{ -1,0,  \allowbreak   +1 \}$-vectors.
By Proposition \ref{PRintpolyIC},
this statement can be rephrased (equivalently)
in terms of integral convexity as follows.

\begin{proposition} \label{PRreconeICgenerator} 
Let $S$ $(\subseteq \ZZ\sp{n})$ be an integrally convex set.
The characteristic cone $C$
of its convex hull $\overline{S}$
is generated by vectors in $\{ -1,0,+1 \}\sp{n}$.
In particular, $C$ is an integer polyhedron.
\end{proposition}
\begin{proof}
Take any $d \in C$ with $\| d \|_{\infty}=1$.
Proposition~\ref{PRconeICgenerator} below
shows that 
there exist $d\sp{1}, d\sp{2}, \ldots , d\sp{h} \in N(d)$
such that 
$d \in \overline{ \{ d\sp{1}, d\sp{2}, \ldots , d\sp{h} \} }$,
where $N(d)$ denotes the integral neighborhood of $d$ defined in \eqref{Nxdef}.
We have
$N(d) \subseteq \{ -1,0,+1 \}\sp{n}$
since $\| d \|_{\infty}=1$.
\end{proof}

\begin{proposition} \label{PRconeICgenerator} 
Let $S$ $(\subseteq \ZZ\sp{n})$ be an integrally convex set,
$x\sp{0} \in S$, and $d \in \RR\sp{n}$ with $\| d \|_{\infty}=1$.
If
\begin{equation} \label{x0lamdbarD}
 x\sp{0} + \lambda d \in \overline{S} 
\quad \mbox{\rm for all $\lambda \geq 0$},
\end{equation}
there exist
$d\sp{1}, d\sp{2}, \ldots , d\sp{h} \in N(d) $ 
such that 
$d \in \overline{ \{ d\sp{1}, d\sp{2}, \ldots , d\sp{h} \} }$ 
and
\begin{equation} \label{x0lamdjD}
 x^0 + k d\sp{j}  \in S 
\qquad (j=1,2,\ldots , h; k=1,2, \ldots ).
\end{equation}
\vspace{-2\baselineskip}\\
\finboxARX
\end{proposition}

\begin{figure}
\centering
\includegraphics[width=0.6\textwidth,clip]{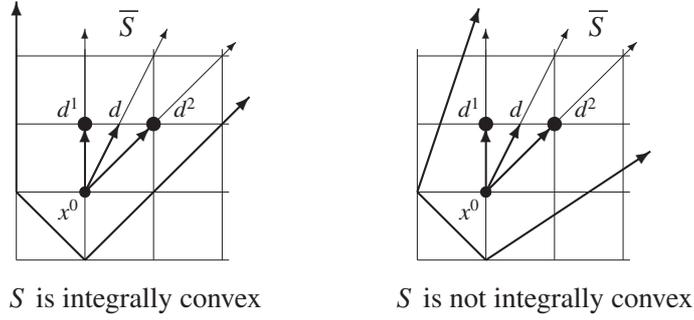}
\caption{Necessity of integral convexity in Proposition~\ref{PRconeICgenerator}}
\label{FGinfdir}
\end{figure}

The condition \eqref{x0lamdbarD} is equivalent to saying that
$d$ belongs to the characteristic cone of $\overline{S}$.
The condition \eqref{x0lamdjD} implies that 
each $d\sp{j}$ belongs to the characteristic cone of $\overline{S}$,
but the converse is not true because
\eqref{x0lamdjD} imposes an additional requirement of integrality.
The role of integral convexity of $S$ is illustrated in Figure \ref{FGinfdir}.
In the left panel, the set $S$ is integrally convex,
while $S$ is not integrally convex in the right,
where $d\sp{1}$ does not meet the condition in \eqref{x0lamdjD}.

To prove Proposition~\ref{PRconeICgenerator},
we need the following general lemma concerning a set 
of $\{ 0,1 \}$-vectors.
Figure~\ref{FGrwbpoint} illustrates this lemma when $X = \{ 0,1 \}\sp{2}$,
where we think of $R$ and $B$ as sets of `red' and `black' points, respectively,
which are disjoint by \eqref{CL2cond0}.

\begin{lemma}\label{LMtam2=1005}
Let $X=\{ 0, 1 \}\sp{m}$.
For any $R \subseteq X$
and $d \in \overline{X} \setminus \overline{R}$,
 there exists some $B \subseteq X$ 
that satisfies the following conditions:
\begin{align} 
& 
R \subseteq X \setminus B  \qquad (\mbox{\rm i.e.,} \ \ R \cap B = \emptyset) ,
\label{CL2cond0}
\\ 
& d \notin \overline{X \setminus B},
\label{CL2cond1}
\\ & \overline{B} \cap (\overline{X \setminus B}) = \emptyset .
\label{CL2cond2}
\end{align} 
Moreover, the elements of $B$ can be ordered 
as $B = \{ d\sp{1}, d\sp{2}, \ldots, d\sp{l} \}$
(where $l=|B|$)
so as to satisfy
\begin{align} 
\overline{ \{ d\sp{1}, d\sp{2}, \ldots, d\sp{i} \}  }
\cap 
\overline{ \{ d\sp{i}, d\sp{i+1}, \ldots, d\sp{l} \}  \cup (X \setminus  B)  }
= \{ d\sp{i} \}
\quad 
\mbox{for $i=1,2,\ldots, l$}. 
\label{CL2cond3}
\end{align} 
\end{lemma}

\begin{figure}
\centering
\includegraphics[width=0.25\textwidth,clip]{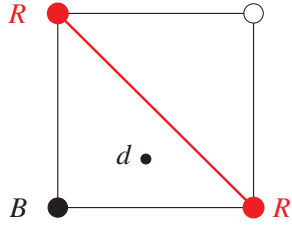}
\caption{Notations in Lemma \ref{LMtam2=1005}}
\label{FGrwbpoint}
\end{figure}

\begin{proof}
We first point out that 
\eqref{CL2cond3} is a refinement of \eqref{CL2cond2}.
Indeed,
\eqref{CL2cond3} for $i = l$ reads
$\overline{B} \cap \overline{ \{ d\sp{l} \} \cup (X \setminus B)} = \{ d\sp{l} \}$.
Since
$d\sp{l} \notin \overline{X \setminus B}$,
this implies
$\overline{B} \cap (\overline{X \setminus B}) = \emptyset$ in 
\eqref{CL2cond2}.
 
In the (special) case where
the given vector $d$ belongs to $X$,
$d$ is an extreme point of $\overline{X}$
and hence we can take $B = \{ d \}$
to meet the requirements
\eqref{CL2cond0},
\eqref{CL2cond1}, and \eqref{CL2cond3}.
In the following we assume $d \notin X$.

The given subset $R$ may be empty or nonempty.
Suppose first that $R \ne \emptyset$.
Since $d \notin \overline{R}$,
the point $d$ can be separated from $\overline{R}$ by a hyperplane.
More precisely, there exists a hyperplane
$H = \{ x \mid a\sp{\top} x = \delta \}$,
where $\delta  \in \RR$,
$a  \in \RR\sp{n}$, and 
$a\sp{\top} x = \sum_{i=1}\sp{n} a_{i} x_{i}$,
such that the (open) half spaces
$H\sp{+} := \{ x \mid  a\sp{\top} x > \delta \}$ and 
$H\sp{-} := \{ x \mid  a\sp{\top} x < \delta \}$
contain $d$ and $\overline{R}$, respectively.
It follows from $d \in H\sp{+}$
and $\overline{R} \subseteq H\sp{-}$ that
$B := H\sp{+} \cap X$ meets the requirements of
\eqref{CL2cond0} and \eqref{CL2cond1}.
Indeed, $B \subseteq X \setminus R$ 
in \eqref{CL2cond0} follows from $B \subseteq H\sp{+}$ and 
$R \subseteq H\sp{-}$, and
$d \notin \overline{X \setminus B}$ in \eqref{CL2cond1} follows from 
$d \in H\sp{+}$ and 
$X \setminus B \subseteq H\sp{-} \cup H$.
To meet \eqref{CL2cond3},
we perturb the vector $a$ so that $a\sp{\top} x$
are distinct for $x \in X \cup \{ d \}$,
and number the elements of $B = \{ d\sp{1}, d\sp{2}, \ldots, d\sp{l} \}$
so that 
$a\sp{\top} d\sp{1} > a\sp{\top} d\sp{2} > \cdots > a\sp{\top} d\sp{l}$.
In the remaining (rather exceptional) case where $R = \emptyset$,
we choose a vector $a$ for which
$a\sp{\top} x$
are distinct for $x \in X \cup \{ d \}$,
and define 
$\delta := a\sp{\top} d  - \varepsilon $
with a sufficiently small positive $\varepsilon$.
Using such $(a, \delta)$ we define
$H$, $H\sp{+}$, $H\sp{-}$, and $B = H\sp{+} \cap X$.
The rest of the argument is the same as in the case of $R \ne \emptyset$.
\end{proof}

We are ready to begin the proof of 
Proposition~\ref{PRconeICgenerator}.
Without loss of generality, we may assume $d \geq \veczero$,
since integral convexity of $S$ is preserved under 
coordinate inversions
$x_{i} \to -x_{i}$ for $i$ in an arbitrary subset of $\{ 1,2,\ldots,n\}$.

Let $X := N(d)$.  Then  $X  \subseteq \{ 0,1 \}\sp{n}$ and $d \in \overline{X}$.
Up to a permutation of coordinates, 
$X$ is equal to a set of the form
$\{ 1 \}\sp{p} \times \{ 0 \}\sp{q} \times \{ 0, 1 \}\sp{m}$
($p+q+m = n$; $p,q,m \geq 0$), so that we may identify $X$ with $\{ 0, 1 \}\sp{m}$.
Define
\begin{equation} \label{setRdef}
 R := \{ d' \in X \mid x^{0} + k d' \in S \  (k=1,2,\ldots)   \},
\end{equation}
or equivalently, $R := X \cap {\rm char.cone}\, \overline{S}$.
Then we have
$d \in \overline{R}$
if and only if there exist
$d\sp{1}, d\sp{2}, \ldots , d\sp{h} \in N(d)$
satisfying 
$d \in \overline{ \{ d\sp{1}, d\sp{2}, \ldots , d\sp{h} \} }$
and
\eqref{x0lamdjD}.
That is, our goal is to show 
$d \in \overline{R}$.
To prove this by contradiction, we assume
$d \notin \overline{R}$.

We have
$R \subseteq X$
and $d \in \overline{X} \setminus \overline{R}$,
where $X$ can be identified with $\{ 0, 1 \}\sp{m}$.
This allows us to use Lemma \ref{LMtam2=1005}
to obtain
$B = \{ d\sp{1}, d\sp{2}, \ldots, d\sp{l} \}$ 
$(\subseteq X)$
satisfying \eqref{CL2cond0}--\eqref{CL2cond3}.
Let
$j\sp{*} \in \{ 0,1,\ldots, l \}$
be the (uniquely determined) number such that
\begin{equation} \label{Prf1=1005}
x\sp{0} + d\sp{i} \notin S  
\quad (i=1,2,\ldots,j\sp{*}),  \qquad  
x\sp{0} + d\sp{j\sp{*}+1} \in S,
\end{equation}
where 
$j\sp{*} =0$ if $x\sp{0} + d\sp{1} \in S$, and
$j\sp{*} =l$ if $x\sp{0} + d\sp{i} \notin S$ for all $i=1,2,\ldots,l$.
Using this index $j\sp{*}$ we define 
$B\sp{*} = \{ d\sp{1}, d\sp{2}, \ldots, d\sp{j\sp{*}} \}$.
Note that $B\sp{*} = \emptyset$ if $j\sp{*} =0$,
and $B\sp{*} = B$ if $j\sp{*} =l$.

Let $y :=  x\sp{0} + d$.
By the assumption \eqref{x0lamdbarD},
namely, 
$d \in {\rm char.cone}\, \overline{S}$,
we have
$y \in \overline{S}$,
which, in turn, implies
$y \in \overline{N(y) \cap S}$
by integral convexity of $S$.
It follows from $N(d) = X$ and the definition of $B\sp{*}$ that%
\footnote{
For any vector $x$ and set $Y$, we use abbreviation $x + Y$ for $\{ x \} + Y$. 
}
\begin{equation*} 
 N(y) \cap  S 
= N(x\sp{0} + d) \cap  S 
= x\sp{0} + \{ d' \in X \mid x\sp{0} +d' \in S \}
\subseteq x\sp{0} + (X \setminus B\sp{*}) .
\end{equation*}
Hence
$y \in \overline{N(y) \cap S}\subseteq  x\sp{0} + \overline{X \setminus B\sp{*}}$,
that is, $d \in \overline{X \setminus B\sp{*}}$.
On the other hand, 
$d \notin \overline{X \setminus B}$ as shown in \eqref{CL2cond1}.
Thus we obtain 
\begin{equation} \label{yinx0XB}
d \in  \overline{X \setminus B\sp{*}}, \qquad
d \notin \overline{X \setminus B}.
\end{equation}
If $B\sp{*} = B$, these two assertion contradict each other,
and we are done.
If $B\sp{*}$ is a proper subset of $B$, 
we cannot derive a contradiction from \eqref{yinx0XB}.

We overcome this difficulty as follows.
Although the definition of $R$ in \eqref{setRdef} refers to $x\sp{0}$,
it is, in fact, independent of the initial point $x\sp{0}$,
as seen from the alternative expression
$R = X \cap {\rm char.cone}\, \overline{S}$.
The set $B$ is also independent of $x\sp{0}$,
whereas $B\sp{*}$, defined via \eqref{Prf1=1005}, 
varies with $x^{0}$, that is, 
$B\sp{*}= B\sp{*}(x^{0})$. 
Our strategy is to show that,
if $B\sp{*}(x^{0}) \ne B$, we can choose another initial point $x^{1}$
satisfying $B\sp{*}(x^{0}) \subsetneqq B\sp{*}(x^{1})$.
By repeating this process, we can increase $B\sp{*}$ until 
$B\sp{*} = B$.  
Then we obtain a contradiction from \eqref{yinx0XB}, 
to complete the proof of Proposition~\ref{PRconeICgenerator}.

\medskip

Since $d\sp{j\sp{*}+1} \in B$
and $R \cap B = \emptyset$
(cf.~\eqref{CL2cond0}),
we have
$d\sp{j\sp{*}+1} \notin R$, 
while
$x^{0} +  d\sp{j\sp{*}+1} \in S$
by \eqref{Prf1=1005}.
Therefore, there exists a positive integer
$k\sp{*} \ge 1$ such that
\begin{equation} \label{Prf1J*=1005}
x^{0} + k d\sp{j\sp{*}+1} \in S \ \ (k=1,2,\ldots, k\sp{*}),
\qquad
x^{0} + (k\sp{*} +1) d\sp{j\sp{*}+1} \notin S.
\end{equation}
This integer $k\sp{*}$ represents the maximum number of steps
from $x^{0}$ toward $d\sp{j\sp{*}+1}$ to stay in $S$.
We define
$x^{1} := x^{0} + k\sp{*} d\sp{j\sp{*}+1}$,
which is a point in $S$.
We shall show 
$B\sp{*}(x^{0}) \subsetneqq B\sp{*}(x^{1})$
by proving
\begin{align} 
& 
x\sp{1} + d\sp{i} \notin S
\quad (i=1,2,\ldots,j\sp{*}), 
\label{Bx1diinS}
\\
& x^{1} + d\sp{j\sp{*}+1} \notin S.
\label{Bx1dj1notinS}
\end{align} 
The second property \eqref{Bx1dj1notinS} is easy to prove. Namely,
\[
x^{1} + d\sp{j\sp{*}+1}
= (x^{0} + k\sp{*} d\sp{j\sp{*}+1})  + d\sp{j\sp{*}+1}
= x^{0} + (k\sp{*}+1) d\sp{j\sp{*}+1} \notin S
\]
using the definition of $k\sp{*}$ in \eqref{Prf1J*=1005}. 
To prove \eqref{Bx1diinS},
we consider a sequence of intermediate points, say, 
$x', x'', \ldots$ 
between $x^{0}$ and $x^{1}$,
where
$x' := x^{0} + d\sp{j\sp{*}+1}$,
$x'' := x^{0} + 2 d\sp{j\sp{*}+1}$, etc.

\begin{claim}  \label{CLxprime}
For $x' = x^{0} + d\sp{j\sp{*}+1}$
we have $x' \in S$ and 
\begin{align} 
x' + d\sp{i} \notin S
\quad (i=1,2,\ldots,j\sp{*}).  
\label{BxprimediinS}
\end{align} 
\end{claim}
\begin{proof}
First, we see $x' \in S$ from \eqref{Prf1J*=1005}. 
To prove \eqref{BxprimediinS}, fix $i$ $(1 \leq i \leq j\sp{*})$ and define 
$\hat d := (d\sp{j\sp{*}+1} + d\sp{i})/2$.
We have
\[
\hat d = (d\sp{j\sp{*}+1} + d\sp{i})/2 \in \overline{ \{ d\sp{j\sp{*}+1}, d\sp{i} \} }
\subseteq \overline{ B\sp{*} \cup \{ d\sp{j\sp{*}+1} \} } .
\]
Since
$\hat d \ne d\sp{j\sp{*}+1}$ 
(which is equivalent to $d\sp{i} \ne d\sp{j\sp{*}+1}$)
and 
\begin{align*}
& \overline{ B\sp{*} \cup \{ d\sp{j\sp{*}+1} \} }
\cap
\overline{ X \setminus B\sp{*}  }
\\
& = \overline{ \{ d\sp{1}, d\sp{2}, \ldots, d\sp{j\sp{*}+1} \}  }
\cap 
\overline{ \{ d\sp{j\sp{*}+1},  \ldots, d\sp{l} \}  \cup (X \setminus  B)  }
\\ & 
= \{ d\sp{j\sp{*}+1} \}
\end{align*}
by \eqref{CL2cond3}, we have
\begin{equation} \label{hatdnotinXB*}
\hat d \notin \overline{ X \setminus B\sp{*}  } .
\end{equation}
On the other hand, 
it follows from the definition of $B\sp{*}$ that
\begin{equation} \label{x0XBNx0dS}
 x^{0} + \overline{ X \setminus B\sp{*}  }
 \  \supseteq \  \overline{ N(x^{0} + \hat d) \cap S } .
\end{equation}
Combining \eqref{hatdnotinXB*} and \eqref{x0XBNx0dS} we obtain
\begin{equation} \label{x0hatdnotinhullN}
x^{0} + \hat d \notin  \overline{ N(x^{0} + \hat d) \cap S  }.
\end{equation}
If $x' + d\sp{i} \in S$ were true,
we would obtain
\begin{equation} \label{x0hatdinhullS}
 x^{0} + \hat d  
 = x^{0} + \frac{1}{2} (d\sp{j\sp{*}+1} + d\sp{i})
 = \frac{1}{2} x^{0} + \frac{1}{2} ( x' + d\sp{i} )
 \in  \overline{S} ,
\end{equation}
which is a contradiction to \eqref{x0hatdnotinhullN},
since $S$ is integrally convex.
Therefore, we must have
$x' + d\sp{i} \notin S$, 
proving \eqref{BxprimediinS}.
\end{proof}

For the second intermediate point
$x''= x^{0} + 2 d\sp{j\sp{*}+1} = x' + d\sp{j\sp{*}+1}$,
we can prove
\begin{align*} 
x'' \in S, \qquad
x'' + d\sp{i} \notin S
\quad (i=1,2,\ldots,j\sp{*})
\end{align*} 
in a similar manner,
by replacing
$(x\sp{0},x')$ 
in the proof of Claim~\ref{CLxprime}
by $(x',x'')$.
Continuing in this way, we can show  
the statement \eqref{Bx1diinS}
at the new initial point $x\sp{1}$
where $B\sp{*}(x^{1})$ is strictly larger than $B\sp{*}(x^{0})$.

If $B\sp{*}(x^{1}) = B$, we are done,
with a contradiction from \eqref{yinx0XB}.
Otherwise, we repeat the same procedure 
to obtain a (finite) sequence 
$x^0, x^1, \ldots x^s$ of initial points 
such that the associated $B\sp{*}$ increases to $B$, i.e.,
$B\sp{*}(x^{0}) 
\subsetneqq B\sp{*}(x^{1})
\cdots \subsetneqq B\sp{*}(x^{s}) = B$.
This completes the proof of Proposition~\ref{PRconeICgenerator}.

\subsection{Proof of Proposition~\ref{PRreconeBI}}
\label{SCproofPRreconeBI}

In this section we prove Proposition~\ref{PRreconeBI},
stating that 
the characteristic cone of a box-integer polyhedron is box-integer.
By Proposition \ref{PRintpolyIC},
this statement can be rephrased (equivalently)
in terms of integral convexity as follows.

\begin{proposition} \label{PRreconeIC}
Let $S$ $(\subseteq \ZZ\sp{n})$ be an integrally convex set.
The characteristic cone $C$
of its convex hull $\overline{S}$
has the property that
$C \cap \ZZ\sp{n}$ is integrally convex.
\finboxARX
\end{proposition}

We begin the proof of 
Proposition~\ref{PRreconeIC}
by observing that
the convex hull $\overline{S}$
can be represented as 
$\overline{S} = Q +  C$
with a bounded box-integer polyhedron $Q$
and a polyhedral cone $C$.
Indeed, by Proposition~\ref{PRdecthmP},
we can decompose $\overline{S}$ as
$\overline{S} = \hat{Q} +  C$,
where $\hat{Q}$ is a polytope and 
$C$ is the characteristic cone of $\overline{S}$.
Take a bounded integral box $B$ containing $\hat{Q}$
and define
$Q := \overline{S} \cap B$,
which is a bounded box-integer polyhedron.
Since
$Q = \overline{S} \cap B = 
(\hat{Q} +  C) \cap B \supseteq \hat{Q} \cap B = \hat{Q}$,
we obtain
$Q + C \supseteq  \hat{Q} + C = \overline{S}$.
The reverse inclusion $Q + C \subseteq \overline{S}$
follows from $Q \subseteq \overline{S}$
and $\overline{S} + C = \overline{S}$ 
(cf.~\eqref{charconePCP}) 
as $Q + C \subseteq \overline{S} + C = \overline{S}$.

\medskip

We prove Proposition~\ref{PRreconeIC} by contradiction.
Namely, we assume that
$C \cap \ZZ\sp{n}$
is not integrally convex and derive 
a contradiction to the integral convexity of $S$.
We shall construct a point 
$y\sp{*} \in \overline{S}$ 
with the property
$y^{*} \notin  \overline{N(y^{*}) \cap  S }$.
We start with an arbitrary
$x^{0} \in Q \cap \ZZ\sp{n}$
and find a point $y^{0} \in x^{0} + C$
with some properties (Claim~\ref{CLy0abc} below).
We consider a system of inequalities
describing $x^{0} + C$.
With reference to the inequalities tight at $y^{0}$,
we find a vertex $x\sp{*}$ of $Q$.
Then the point $y\sp{*}$ is constructed as
$y\sp{*} = y^{0} + (x\sp{*} - x^{0})$
in \eqref{y*Defy0x*x0} below.

Recalling that $Q$ is a nonempty integer polyhedron,
take any $x^{0} \in Q \cap \ZZ\sp{n}$ and define
\[
D := x^{0} + C,  \qquad
D_I := D \cap \ZZ\sp{n}.
\]
By Proposition~\ref{PRreconeICgenerator},
$C$ is an integer polyhedron,
which implies that
$D$ is an integer polyhedron and $D = \overline{D_{I}}$.
The set $D_I$ is not integrally convex
as a consequence of the assumption that 
$C \cap \ZZ\sp{n}$ is not integrally convex.

\begin{claim}  \label{CLy0abc}
There exists $y^{0} \in D$ that satisfies the following conditions:
\begin{align} 
& y^{0} \notin \overline{N(y^{0}) \cap D_{I}},
\label{CLy0abcA}
\\ 
& \mbox{\rm $y^{0}$ is a vertex of $\overline{N(y^{0})} \cap D$},
\label{CLy0abcB}
\\
& \mbox{\rm $y^{0}$ is a relative interior point of $\overline{N(y^{0})}$}.
\label{CLy0abcC}
\end{align} 
\end{claim}

\begin{proof}
Since $D_I$ is not integrally convex,
there exists
$z \in \overline{D_{I}}$
such that $z \notin \overline{N(z) \cap D_{I}}$.
Take such $z$ 
with the smallest dimension of $\overline{N(z)}$.
Note that $\overline{N(z)}$ is an integral box of 
the form 
$\{ x \in \RR\sp{n} \mid l   \leq x \leq  u \}$
for some 
$l , u \in \ZZ\sp{n}$ with 
$\| u - l \, \|_{\infty} \leq 1$
and the dimension of $\overline{N(z)}$ is equal to
the number of indices $i$ satisfying  
$u_{i} - l_{i} = 1$.

\begin{figure}
\centering
\includegraphics[width=0.31\textwidth,clip]{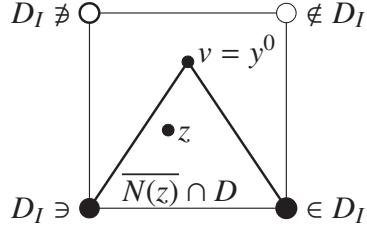}
\caption{Notations in the proof of Claim~\ref{CLy0abc}
(The square represents $\overline{N(z)}$ 
and the triangle is $\overline{N(z)} \cap D$)}
\label{FGdecNy0Dv}
\end{figure}

The set $\overline{N(z)} \cap D$ is a bounded polyhedron, and
$( \overline{N(z)} \cap D ) \setminus \overline{N(z) \cap D_{I}} \ne \emptyset$
since $z \in  ( \overline{N(z)} \cap D ) \setminus \overline{N(z) \cap D_{I}}$.
Hence there is a vertex $v$ of 
$\overline{N(z)} \cap D$
not contained in $\overline{N(z) \cap D_{I}}$
(see Figure \ref{FGdecNy0Dv}).
The vertex $v$ is a relative interior point of $\overline{N(z)}$,
because, otherwise,
we would have
$\dim \overline{N(v)} < \dim \overline{N(z)}$
while
$v \notin \overline{N(v) \cap D_{I}}$ from 
$v \notin \overline{N(z) \cap D_{I}} \supseteq \overline{N(v) \cap D_{I}}$,
a contradiction to our choice of $z$.
Since $v$ is a relative interior point of $\overline{N(z)}$,
we have $N(v) = N(z)$.
Let $y^{0}:=v$, which satisfies the three conditions
\eqref{CLy0abcA}--\eqref{CLy0abcC}.
\end{proof}

Consider a (non-redundant) system of inequalities
describing $D$.
Since 
$y^{0}$ 
is a vertex of $\overline{N(y^{0})} \cap D$
lying in the relative interior of 
$\overline{N(y^{0})}$,
at least one inequality is tight (i.e., satisfied in equality).
Enumerate all such inequalities as
\begin{equation} \label{Tam2IneqSys}
 a_{1}\sp{\top} x \le \beta_{1}, \ \ 
 a_{2}\sp{\top} x \le \beta_{2}, \ \ 
 \ldots , \ \ 
 a_{k}\sp{\top} x \le \beta_{k},
\end{equation} 
where $k \ge 1$.
By definition we have
$a_{i}\sp{\top} y^{0} = \beta_{i}$ for $i=1,2,\ldots,k$.
Since $D = x^{0} +C$
and $C$ is a cone,
all the inequalities in \eqref{Tam2IneqSys} are also tight at $x^{0}$, that is,
$a_{i}\sp{\top} x^{0} = \beta_{i}$ for $i=1,2,\ldots,k$.

\begin{claim}  \label{CLexistmu}
There exist some positive coefficients
$\mu_{1}, \mu_{2}, \ldots, \mu_{k} > 0$
such that 
$a = \sum_{i=1}\sp{k} \mu_{i} a_{i}$
and $\beta = \sum_{i=1}\sp{k} \mu_{i} \beta_{i}$
satisfy
\begin{align} 
& 
a\sp{\top} x^{0} = \beta, \quad 
a\sp{\top} y^{0} = \beta,
\label{Ineqx0y0} \\
& 
a\sp{\top} x  \leq  \beta
\qquad (\forall x \in  D),
\label{IneqInD} \\
& 
a\sp{\top} x  \leq  0
\qquad (\forall x \in C),
\label{IneqInC} \\
& a\sp{\top} x \ne \beta
\qquad (\forall x \in N(y^{0}) \setminus D_{I}).
\label{IneqNotDI}
\end{align}
\end{claim}

\begin{proof}
\eqref{Ineqx0y0}--\eqref{IneqInC}
hold for any  $\mu_{1}, \mu_{2}, \ldots, \mu_{k} >0$.
\eqref{Ineqx0y0} is immediate from the tightness
$a_{i}\sp{\top} x^{0} = a_{i}\sp{\top} y^{0} = \beta_{i}$
for $i=1,2,\ldots,k$.
\eqref{IneqInD} holds since
the inequalities in \eqref{Tam2IneqSys} are valid for $D$. 
\eqref{IneqInC} follows from 
\eqref{Ineqx0y0} and \eqref{IneqInD}
because $D=x^{0}+C$ and $C$ is a cone.
\eqref{IneqNotDI} can be shown as follows.
Since $y^{0}$ is a vertex, 
the intersection of 
$\overline{N(y^{0})}$ and
the hyperplanes 
$a_{i}\sp{\top} x = \beta_{i}$
$(i=1,2,\ldots,k)$
consists of a single vector $y^{0}$,
that is,
for $x \in \overline{N(y^{0})}$, we have
$a_{i}\sp{\top} x = \beta_{i}$ for all $i=1,2,\ldots,k$
if and only if $x  = y^{0}$.
Therefore, for each $x \in N(y^{0}) \setminus D_{I}$,
there is some $i$ with
$a_{i}\sp{\top} x \ne \beta_{i}$.
It then follows that
\eqref{IneqNotDI} holds
for randomly chosen
$\mu_{1}, \mu_{2}, \ldots, \mu_{k} > 0$.
\end{proof}

\begin{claim}  \label{CLtam3}
\begin{equation}  \label{Tam3y0notinLargerSet}
  y^{0} \notin 
 \overline{N(y^{0}) \cap  \{  x \in \ZZ\sp{n} \mid 
 a\sp{\top} x \le \beta \} }.
\end{equation}
\end{claim}
\begin{proof}
Recall from \eqref{IneqInD} that
$\{  x \in \ZZ\sp{n} \mid  a\sp{\top} x \le \beta \} \supseteq D_{I}$.
Using  notation
$E := \{  x \in \ZZ\sp{n} \mid  a\sp{\top} x \le \beta \} \setminus D_{I}$,
we can rewrite \eqref{Tam3y0notinLargerSet} as
\begin{equation}  \label{Tam3y0notinLargerSet2}
  y^{0} \notin \overline{N(y^{0}) \cap  (E \cup D_{I}) }
\ 
 = \overline{(N(y^{0}) \cap  D_{I}) \cup (N(y^{0}) \cap  E) }.
\end{equation}
We have
$y^{0} \notin \overline{N(y^{0}) \cap  D_{I} }$ 
in \eqref{CLy0abcA} and 
$a\sp{\top} y^{0} = \beta$ in \eqref{Ineqx0y0}, 
whereas
$a\sp{\top} x < \beta$
for all $x \in N(y^{0}) \cap E$
by \eqref{IneqNotDI}.
Then \eqref{Tam3y0notinLargerSet2} follows.
\end{proof}

Let $\beta\sp{*}$ denote the maximum value of $a\sp{\top} x $ over $Q$, 
that is, 
\begin{equation} \label{beta*def}
 \beta\sp{*} := \max \{ a\sp{\top} x \mid x \in Q  \}.
\end{equation}
Since $Q$ is a bounded integer polyhedron,
we may assume that this maximum is attained by an integer vector
$x\sp{*} \in Q \cap \ZZ\sp{n}$.
Define $y^{*} \in \RR\sp{n}$ by
\begin{equation}  \label{y*Defy0x*x0}
y\sp{*} := y^{0} + (x\sp{*} - x^{0}) .
\end{equation}
We have
$y\sp{*} \in \overline{S}$,
since
$y\sp{*} = x\sp{*} + (y^{0}  - x^{0})  \in Q + (D -  x^{0}) = Q+C =\overline{S}$.

\begin{claim}  \label{CLtam4}
\begin{equation}  \label{Tam7y0notinLargerSet}
  y^{*} \notin 
 \overline{N(y^{*}) \cap  
\{  x \in \ZZ\sp{n} \mid a\sp{\top} x \le \beta\sp{*}  \} } .
\end{equation}
\end{claim}
\begin{proof}
Recall
from \eqref{Tam3y0notinLargerSet}
that
\[ 
  y^{0} \notin 
 \overline{N(y^{0}) \cap  \{  x \in \ZZ\sp{n} \mid 
 a\sp{\top} x \le \beta \} }.
\] 
By adding 
$x\sp{*} - x^{0}$
to the left-hand side, we obtain
$y\sp{*} = y^{0} + (x\sp{*} - x^{0})$.
On the right-hand side,
we have
$N(y^{0}) + (x\sp{*} - x^{0}) = N(y^{0} + x\sp{*} - x^{0}) = N(y^{*})$,
where the first equality is true by $x\sp{*} - x^{0} \in \ZZ\sp{n}$.
Since
$a\sp{\top} x\sp{*} = \beta\sp{*}$ by the definition of $x\sp{*}$ and
$a\sp{\top} x\sp{0} = \beta$ by \eqref{Ineqx0y0},
we also have
\begin{align*} 
\{  x  \mid a\sp{\top} x \le \beta  \}  + (x\sp{*} - x^{0})
& = \{  x + (x\sp{*} - x^{0})  \mid a\sp{\top} x \le \beta  \}  
\\ & = \{  z  \mid a\sp{\top} (z - x\sp{*} + x^{0}) \le \beta  \}  
= \{  z   \mid a\sp{\top} z \le \beta\sp{*}  \} .
\end{align*} 
Thus we obtain \eqref{Tam7y0notinLargerSet}.
\end{proof}

\begin{claim}  \label{CLtam5}
\begin{equation}  \label{Tam5Sincluded}
 S \subseteq \{  x \in \ZZ\sp{n} \mid a\sp{\top} x \le \beta\sp{*}  \} .
\end{equation}
\end{claim}
\begin{proof}
We have 
$Q \subseteq \{ x \mid a\sp{\top} x \leq \beta\sp{*}  \}$
by the definition \eqref{beta*def} of $\beta\sp{*}$,
whereas
$C \subseteq \{ x \mid a\sp{\top} x \leq 0  \}$
by \eqref{IneqInC}.
Therefore, every $x \in \overline{S} = Q + C$
satisfies $a\sp{\top} x \leq \beta\sp{*}$.
\end{proof}

It follows from 
\eqref{Tam7y0notinLargerSet}
and \eqref{Tam5Sincluded}
that
$y^{*} \notin  \overline{N(y^{*}) \cap  S }$,
whereas
$y\sp{*} \in \overline{S}$.
This is a contradiction to the integral convexity of $S$,
completing the proof of Proposition~\ref{PRreconeIC}.

\subsection{Proof of Theorem~\ref{THicsetDecZ}}
\label{SCproofTHicsetDecZ}

In this section we prove Theorem~\ref{THicsetDecZ},
stating that 
every integrally convex set $S$ 
$(\subseteq \ZZ\sp{n})$
can be represented as 
$S = T + G$ 
with a bounded integrally convex set $T$ and a conic integrally convex set $G$.

By Proposition~\ref{PRdecthmP},
the convex hull $\overline{S}$ of $S$
can be represented as 
\begin{equation} \label{prf1SbarQhatC=1105}
\overline{S} = \hat{Q} + C 
\end{equation}
with a polytope $\hat{Q}$ and 
the characteristic cone $C$ of $\overline{S}$.
By Proposition~\ref{PRreconeIC},
$C \cap \ZZ\sp{n}$ is integrally convex.
With reference to the polytope $\hat{Q}$, define
\[
 l_{i}  := 
\big\lfloor \, \min \{ x_{i} \mid x \in \hat{Q} \} \,  \big\rfloor ,
\qquad
 u_{i}  := 
 \big\lceil \, \max \{ x_{i} \mid x \in \hat{Q} \} \, \big\rceil 
\]
for $i=1,2,\ldots,n$.
The numbers $l_{i}, u_{i}$ are (finite) integers
with $l_{i} \leq  u_{i}$,
since $\hat{Q}$ is a nonempty and bounded polyhedron.

Let $\{ d\sp{1},  d\sp{2}, \ldots,  d\sp{L}  \}$
be a generating set of cone $C$,
where we may assume
$d\sp{j} \in \{ -1,0,+1 \}\sp{n}$ 
by Proposition~\ref{PRreconeICgenerator}.
With reference to the number $L$ of the generators of $C$,
define a bounded integral box $B$ by 
\begin{equation*} 
 B:= \{ x \in \RR\sp{n}  \mid  l_{i} - L \leq x_{i} \leq u_{i} + L
\ (i=1,2,\ldots, n)  \}
\end{equation*}
and put $Q := \overline{S} \cap B$,
which is a bounded box-integer polyhedron containing $\hat{Q}$.
We have
$\overline{S} = Q + C$, since
$\overline{S} = \hat{Q} +  C   \subseteq Q + C \subseteq \overline{S} + C =  \overline{S}$.

Define
\begin{equation} \label{prf34QIdef=1105}
T := Q \cap \ZZ\sp{n}
= \overline{S} \cap B \cap \ZZ\sp{n} = S \cap B ,
\qquad
G := C \cap \ZZ\sp{n},
\end{equation}
which are, respectively,  a bounded integrally convex set
and a conic integrally convex set.
In the following we show $S = T + G$
by a sequence of claims.

\begin{claim}  \label{CLSsupsetQConZ=1105}
$S \supseteq T + G$.
\end{claim}
\begin{proof}
For two (discrete) sets
$S_{1}, S_{2} \subseteq \ZZ\sp{n}$,
in general, we have
\[ 
(\overline{ S_{1} } + \overline{ S_{2} }) \cap \ZZ\sp{n}
\supseteq  S_{1}  +  S_{2} .
\] 
Using this for $(S_{1}, S_{2}) = (T,G)$ 
as well as
$\overline{S} = Q + C = \overline{T} + \overline{G}$,
we obtain
\[
S = \overline{S} \cap \ZZ\sp{n}
= (\overline{T} + \overline{G}) \cap \ZZ\sp{n}
\supseteq  T + G.
\]
\vspace{-2\baselineskip}\\
\end{proof}

To show the reverse inclusion $S \subseteq T + G$, take any $z \in S$.
By $\overline{S} = \hat{Q} + C$ in \eqref{prf1SbarQhatC=1105}, there exist
real vectors
$\hat{x} \in \hat{Q}$ and $\hat d \in C$ 
satisfying
\[ 
 z = \hat{x} + \hat d .
\] 
The vector $\hat d$ can be represented 
as a nonnegative combination of
the generators 
$\{ d\sp{1},  d\sp{2}, \ldots,  d\sp{L}  \}$
of $C$ as
\[ 
  \hat d = \sum_{j=1}\sp{L} \lambda_{j} d\sp{j},
\qquad \lambda_{j} \geq 0 \ \ (j=1,2,\ldots, L) .
\] 
With reference to this expression, define vectors $d\sp{*}$ and $x\sp{*}$ by  
\begin{align} 
  d\sp{*} & := \sum_{j=1}\sp{L} \lfloor \lambda_{j} \rfloor \  d\sp{j},
\label{prf6ddef=1105}
\\
  x\sp{*} & := \hat{x} + 
  \sum_{j=1}\sp{L} (\lambda_{j} - \lfloor \lambda_{j} \rfloor) \  d\sp{j},
\label{prf6xdef=1105}
\end{align} 
for which we have
\begin{equation} \label{prf7zxstardstar=1105}
 x\sp{*}  + d\sp{*}  = \hat{x} + \hat d = z .
\end{equation}

\begin{claim}  \label{CLdinCI=1105}
$d\sp{*} \in G$.
\end{claim}
\begin{proof}
\eqref{prf6ddef=1105} shows $d\sp{*} \in C$.
We also have $d\sp{*} \in \ZZ\sp{n}$,
since 
$\lfloor \lambda_{j} \rfloor \in \ZZ$
and $d\sp{j} \in \{ -1,0,+1 \}\sp{n}$ 
for $j=1,2,\ldots, L$
by Proposition~\ref{PRreconeICgenerator}.
Therefore, $d\sp{*} \in C \cap \ZZ\sp{n} =  G$.
\end{proof}

\begin{claim}  \label{CLxinQI=1105}
$x\sp{*} \in T$.
\end{claim}
\begin{proof}
Since
$T = \overline{S} \cap B \cap \ZZ\sp{n}$
(see \eqref{prf34QIdef=1105}),
it suffices to show
(i) $x\sp{*} \in \ZZ\sp{n}$,
(ii) $x\sp{*} \in \overline{S}$, and
(iii) $x\sp{*} \in B$.
We have
$x\sp{*} \in \ZZ\sp{n}$,
since
$z \in \ZZ\sp{n}$,
$d\sp{*} \in \ZZ\sp{n}$, and
$x\sp{*}  = z - d\sp{*}$
by \eqref{prf7zxstardstar=1105}.
We have $x\sp{*} \in \overline{S}$,
since $x\sp{*} \in \hat{Q} + C$ by \eqref{prf6xdef=1105} and
$\hat{Q} + C = \overline{S}$ by \eqref{prf1SbarQhatC=1105}.
Finally,
we show $x\sp{*} \in B$.
For the first term $\hat{x}$ 
on the right-hand side of \eqref{prf6xdef=1105},
we have $l \leq \hat{x} \leq u$ since 
$\hat{x} \in \hat{Q}$.
Each component of the second term 
$ \sum_{j=1}\sp{L} (\lambda_{j} - \lfloor \lambda_{j} \rfloor) \  d\sp{j}$
lies between
$-L$ and $+L$,
since
$0  \leq \lambda_{j} - \lfloor \lambda_{j} \rfloor < 1$
and
$d\sp{j} \in \{ -1,0,+1 \}\sp{n}$
for $j=1,2,\ldots, L$
by Proposition~\ref{PRreconeICgenerator}.
Therefore, $x\sp{*} \in B$.
\end{proof}

The inclusion $S \subseteq T + G$
follows from 
Claims \ref{CLdinCI=1105} and \ref{CLxinQI=1105},
while $S \supseteq T + G$
is already shown in Claim~\ref{CLSsupsetQConZ=1105}.
This completes the proof of Theorem~\ref{THicsetDecZ}.

\subsection{Proof of Theorem \ref{THlnatmnatDecR}}
\label{SCproofTHlnatmnatDecR}

In this section we prove Theorem~\ref{THlnatmnatDecR} for polyhedra
$P$ $(\subseteq \RR\sp{n})$ 
with particular discrete convexities
such as \Lnat-convexity, \Mnat-convexity, etc.
The proof for the case (1) of \Lnat-convex polyhedra has already 
been given in Section~\ref{SCresultBI},  
right after Theorem~\ref{THlnatmnatDecR}.
Here we present a unified proof scheme
for all cases including \Lnat-convex polyhedra.
We use a generic name ``A-convex''
to mean any of 
\Lnat-convex, \LLnat-convex, \Mnat-convex, M-convex,  \MMnat-convex, \MM-convex, 
and multimodular.

The unified proof scheme is as follows.
Let $P$ be an A-convex polyhedron.
By Proposition~\ref{PRdecthmP},
we can decompose $P$ as
$P = \hat{Q} +  C$,
where $\hat{Q}$ is a polytope and 
$C$ is the characteristic cone of $P$.
We assume that
\begin{align} 
& \mbox{The characteristic cone of an A-convex polyhedron is A-convex.}
\label{prfschR1} 
\\
& \mbox{There exists a bounded A-convex polyhedron $Q$ satisfying 
$\hat{Q} \subseteq Q \subseteq P$.}
\label{prfschR2} 
\end{align}
By $\hat{Q} \subseteq Q \subseteq P$ in \eqref{prfschR2}, we have
$P = \hat{Q} +  C   \subseteq Q + C \subseteq P + C =  P$.
This shows $P = Q + C$, 
where 
$Q$ is a bounded A-convex polyhedron by \eqref{prfschR2}
and 
$C$ is an A-convex cone by \eqref{prfschR1}.

The first assumption \eqref{prfschR1}
is met by each discrete convexity in (1)--(5).
Indeed, a polyhedron $P$ with such discrete convexity
can be described as
$P = \{ x \mid Ax \leq b \}$,
where a necessary and sufficient condition on $(A,b)$ 
for that discrete convexity of $P$ is known.
For example, 
an \Lnat-convex polyhedron is described by \eqref{ineqLnat} 
and an \Mnat-convex polyhedron by \eqref{ineqMnat};
see Murota \cite{Mdcasiam}, 
Moriguchi--Murota \cite[Table 1]{MM22L2poly}, and
Murota--Tamura \cite[Table 1]{MT23ICsurv}
for other cases.
This enables us to prove that
the characteristic cone 
$C = \{ d \mid A d \leq 0 \}$
is also endowed with 
the same kind of discrete convexity.

For the second assumption \eqref{prfschR2},
we consider $Q := P \cap B$
for a bounded integral box $B$ containing $\hat{Q}$,
expecting that $Q$ is endowed with A-convexity
as a consequence of the assumed A-convexity of $P$.
This construction is indeed valid 
for all discrete convexities in question, 
with the exception of \LLnat-convexity in (2)
(see Remark~\ref{RMinterL2boxR} below).

In Case (2) of an \LLnat-convex polyhedron $P$,
we construct
an \LLnat-convex  (integer) polyhedron $Q$ as follows.
Let $P = P_{1} + P_{2}$ with two \Lnat-convex polyhedra 
$P_{1}$ and $P_{2}$.
Enumerate all vertices of the polytope $\hat{Q}$ 
as $\{ z\sp{1},z\sp{2},\ldots,z\sp{m} \}$,
where $z\sp{j} \in \RR\sp{n}$ for $j=1,2,\ldots,m$. 
By $z\sp{j} \in \hat{Q} \subseteq P_{1} + P_{2}$,
each $z\sp{j}$
can be expressed as 
$z\sp{j} = x\sp{j} + y\sp{j}$ 
with $x\sp{j} \in P_{1}$
and $y\sp{j} \in P_{2}$.
Take integral boxes $B_{1}$ and $B_{2}$ satisfying
$\{ x\sp{1},x\sp{2},\ldots,x\sp{m} \} \subseteq B_{1}$
and $\{ y\sp{1},y\sp{2},\ldots,y\sp{m} \} \subseteq B_{2}$,
respectively,
and define
$Q_{1} := P_{1} \cap B_{1}$,
$Q_{2} := P_{2} \cap B_{2}$,
and $Q := Q_{1} + Q_{2}$.
Then $Q_{1}$ and $Q_{2}$ are \Lnat-convex (integer) polyhedra,
and hence $Q$ is an \LLnat-convex  (integer) polyhedron.
Then we have
\begin{align*} 
\hat{Q} &= 
\overline{ \{ z\sp{1},z\sp{2},\ldots,z\sp{m} \} }
=
\overline{ \{ x\sp{1}+y\sp{1},x\sp{2}+y\sp{2},\ldots,x\sp{m}+y\sp{m} \} }
\\ &
\subseteq 
\overline{ \{ x\sp{1},x\sp{2},\ldots,x\sp{m} \} }
+
\overline{ \{ y\sp{1},y\sp{2},\ldots,y\sp{m} \} }
\\ & 
\subseteq 
(P_{1} \cap B_{1}) + (P_{2} \cap B_{2}) 
 = Q_{1} + Q_{2} = Q
\end{align*} 
and
$Q = Q_{1} + Q_{2} \subseteq P_{1} + P_{2} = P$.
Thus we obtain
$\hat{Q} \subseteq Q \subseteq P$.

This completes the proof of Theorem~\ref{THlnatmnatDecR}.

\begin{remark} \rm \label{RMinterL2boxR}
The intersection of an \LLnat-convex polyhedron
with an integral box is not necessarily \LLnat-convex.
For example,
let $P_{1}$ be the line segment connecting $(0, 0, 0)$ and $(1, 1, 0)$
and
$P_{2}$ be the one connecting $(0, 0, 0)$ and $(0, 1, 1)$.
Then 
$P  = P_{1} + P_{2}$
 $(\subseteq \RR\sp{3})$
is an \LLnat-convex polyhedron, which is a parallelogram 
lying on the plane $x_{1} - x_{2} + x_{3} =0$
in $\RR\sp{3}$.
For the unit box
$B = \{ x \mid 0 \leq x_{i} \leq 1 \ (i=1,2,3) \}$, 
the intersection $P \cap B$ 
is a triangle with vertices at $(0, 0, 0)$, $(0, 1, 1)$, and $(1, 1, 0)$.
This triangle is not \LLnat-convex.
\finbox
\end{remark}

\begin{remark} \rm \label{RMproofL2R}
Here is an alternative proof of Theorem~\ref{THlnatmnatDecR}(2) 
that relies on (1) for an \Lnat-convex polyhedron.
Let $P = P_{1} + P_{2}$ with \Lnat-convex polyhedra $P_{1}$ and $P_{2}$.
By (1) we have
$P_{i} = Q_{i} + C_{i}$ 
with a bounded \Lnat-convex polyhedron $Q_i$ and an \Lnat-convex cone $C_{i}$,
where $i=1,2$.
Then $P = (Q_{1} + Q_{2}) + (C_{1} + C_{2})$, where $Q_{1} + Q_{2}$ 
is a bounded \LLnat-convex polyhedron
and $C_{1} + C_{2}$ is an \LLnat-convex cone.
\finbox
\end{remark}

\subsection{Proof of Theorem~\ref{THlnatmnatDecZ}}
\label{SCproofTHlnatmnatDecZ}

In this section we prove Theorem~\ref{THlnatmnatDecZ} for discrete sets
$S$ $(\subseteq \ZZ\sp{n})$ 
with particular discrete convexities
such as \Lnat-convexity, \Mnat-convexity, etc.
The proof relies on Theorem~\ref{THicsetDecZ}
for integrally convex sets.
Just as in Section~\ref{SCproofTHlnatmnatDecR},
we present a unified proof scheme
by using a generic name ``A-convex''
to mean any of 
\Lnat-convex, \LLnat-convex, \Mnat-convex, M-convex,  \MMnat-convex, \MM-convex, 
and multimodular.

The unified proof scheme is as follows.
Let $S$ be an A-convex set.
This implies that $S$ is an integrally convex set. 
By Theorem~\ref{THicsetDecZ}
we can decompose $S$ as
$S = \hat{T} +  G$,
where $\hat{T}$ is a bounded integrally convex set and 
$G$ is a conic integrally convex set.
We have
$G = C \cap \ZZ\sp{n}$
for the characteristic cone $C$ of 
the convex hull $\overline{S}$ of $S$,
where $\overline{S}$ is an A-convex polyhedron.
We assume that
\begin{align} 
& \mbox{The characteristic cone of an A-convex polyhedron is A-convex.}
\label{prfschZ1} 
\\
& \mbox{There exists a bounded A-convex set $T$ satisfying 
$\hat{T} \subseteq T \subseteq S$.}
\label{prfschZ2} 
\end{align}
By $\hat{T} \subseteq T \subseteq S$ in \eqref{prfschZ2}, we have
$S = \hat{T} +  G   \subseteq T + G \subseteq S + G \subseteq  S$,
where the last inclusion follows from
$ \overline{S + G} = \overline{S} + \overline{G} 
= \overline{S} + C = \overline{S}$, \ 
$S + G \subseteq  \overline{S + G} \cap \ZZ\sp{n}$, \ 
and $\overline{S} \cap \ZZ\sp{n} = S$.
Therefore, $S = T + G$,
where
$T$ is a bounded A-convex set by \eqref{prfschZ2}
and 
$G$ is a conic A-convex set by \eqref{prfschZ1}.

The first assumption \eqref{prfschZ1},
which is the same as \eqref{prfschR1},
is met by each discrete convexity in (1)--(5),
as explained in the proof of Theorem \ref{THlnatmnatDecR} 
in Section~\ref{SCproofTHlnatmnatDecR}.
Recall that the inequality representations are used here.

For the second assumption \eqref{prfschZ2},
we consider $T := S\cap B$
for a bounded integral box $B$ containing $\hat{T}$,
expecting that $T$ is endowed with A-convexity
as a consequence of the assumed A-convexity of $S$.
This construction is indeed valid
for all discrete convexities in question
(see \cite{Mdcasiam,Msurvop21}),
with the exception of \LLnat-convexity in (2)
(see Remark~\ref{RMinterL2boxZ} below).

In Case (2) of an \LLnat-convex set $S$, we construct $T$ as follows.
Represent $S$ as $S = S_{1} + S_{2}$ with two \Lnat-convex sets 
$S_{1}$ and $S_{2}$.
Enumerate all members of the finite set $\hat{T}$ 
as $\hat{T} = \{ z\sp{1},z\sp{2},\ldots,z\sp{m} \}$.
Each $z\sp{j} \in \hat{T} \subseteq S = S_{1} + S_{2}$
can be expressed as 
$z\sp{j} = x\sp{j} + y\sp{j}$ 
with $x\sp{j} \in S_{1}$
and $y\sp{j} \in S_{2}$.
Take integral boxes $B_{1}$ and $B_{2}$ satisfying
$\{ x\sp{1},x\sp{2},\ldots,x\sp{m} \} \subseteq B_{1}$
and $\{ y\sp{1},y\sp{2},\ldots,y\sp{m} \} \subseteq B_{2}$,
respectively,
and define
$T_{1} := S_{1} \cap B_{1}$,
$T_{2} := S_{2} \cap B_{2}$,
and $T := T_{1} + T_{2}$.
Then $T_{1}$ and $T_{2}$ are \Lnat-convex,
and hence $T$ is \LLnat-convex.
We have
$\hat{T} \subseteq T$,
since 
$x\sp{j} \in S_{1} \cap B_{1}$ and 
$y\sp{j} \in S_{2} \cap B_{2}$
imply that
$z\sp{j} = x\sp{j} + y\sp{j} \in 
(S_{1} \cap B_{1} ) + (S_{2} \cap B_{2}) = T_{1} + T_{2} = T$.
Finally we note 
$T = T_{1} + T_{2} \subseteq S_{1} + S_{2} = S$,
to obtain $\hat{T} \subseteq T \subseteq S$.

This completes the proof of Theorem~\ref{THlnatmnatDecZ}.

\begin{remark} \rm \label{RMinterL2boxZ}
The intersection of an \LLnat-convex set 
with an integral box is not necessarily \LLnat-convex.
For example, consider an \LLnat-convex set 
$S = S_{1} + S_{2}$ 
given by two \Lnat-convex sets
$S_{1} =  \{(0, 0, 0), (1, 1, 0)\}$
and
$S_{2} = \{(0, 0, 0)$, $(0, 1, 1)\}$.
That is,
$S  = S_{1} + S_{2} = \{(0, 0, 0), \ (0, 1, 1),  \  \allowbreak   (1, 1, 0), \ (1, 2, 1)\}$.
For $B = \{ x \mid 0 \leq x_{i} \leq 1 \ (i=1,2,3) \}$ 
we have
$S \cap B = \{(0, 0, 0), \ (0, 1, 1), \  (1, 1, 0) \}$,
which is not \LLnat-convex.
\finbox
\end{remark}

\begin{remark} \rm \label{RMproofL2Z}
Here is an alternative proof of Theorem~\ref{THlnatmnatDecZ}(2) 
that relies on (1) for an \Lnat-convex set.
Let $S = S_{1} + S_{2}$ with \Lnat-convex sets $S_{1}$ and $S_{2}$.
By (1) we have
$S_{i} = T_{i} + G_{i}$ 
with a bounded \Lnat-convex set $T_i$ and a conic \Lnat-convex set $G_{i}$,
where $i=1,2$.
Then $S = (T_{1} + T_{2}) + (G_{1} + G_{2})$, where $T_{1} + T_{2}$ 
is a bounded \LLnat-convex set
and $G_{1} + G_{2}$ is a conic \LLnat-convex set.
Note that
$\overline{G_{1} + G_{2}} = 
\overline{G_{1}} + \overline{G_{2}}$ is a cone.
\finbox
\end{remark}

\section{Conclusion}
\label{SCconcl}

Our proofs given in Sections \ref{SCproofrecone}--\ref{SCproofTHicsetDecZ}
are long and primitive based on the very definition of integral convexity.
On the other hand, it is known
 (Chervet--Grappe--Robert \cite{CGR21})
that a polyhedral cone is box-integer if and only if it is box-TDI.
It is left for future investigation to find shorter or more transparent proofs,
possibly making use of this equivalence.

\bigskip

\noindent {\bf Acknowledgement}.
This work was supported by JSPS/MEXT KAKENHI JP23K11001  and JP21H04979.



\newpage
\tableofcontents

\end{document}